\documentclass[reqno,a4paper]{amsart}
\usepackage{amssymb}
\usepackage{amsmath}
\newcommand{\half}{\frac{1}{2}}
\newcommand{\abs}[1]{\vert #1 \vert}
\newcommand{\norm}[1]{\left\Vert #1 \right\Vert}
\newcommand{\R}{\mathbb{R}}

\begin{document} 
\newtheorem{prop}{Proposition}[section]
\newtheorem{Def}{Definition}[section]
\newtheorem{theorem}{Theorem}[section]
\newtheorem{lemma}{Lemma}[section]
 \newtheorem{Cor}{Corollary}[section]

\title[LWP for Yang-Mills-Dirac]{\bf Local well-posedness of the coupled Yang-Mills and Dirac system in temporal gauge}
\author[Hartmut Pecher]{
{\bf Hartmut Pecher}\\
Fakult\"at f\"ur  Mathematik und Naturwissenschaften\\
Bergische Universit\"at Wuppertal\\
Gau{\ss}str.  20\\
42119 Wuppertal\\
Germany\\
e-mail {\tt pecher@uni-wuppertal.de}}
\date{}

\begin{abstract}
We consider the classical Yang-Mills system coupled with a Dirac equation in 3+1 dimensions in temporal gauge. Using that most of the nonlinear terms fulfill a null condition we prove local well-posedness for small data with minimal regularity assumptions. This problem for smooth data was solved forty years ago by Y. Choquet-Bruhat and D. Christodoulou. The corresponding problem in Lorenz gauge was considered recently by the author in \cite{P1}.
\end{abstract}
\maketitle
\renewcommand{\thefootnote}{\fnsymbol{footnote}}
\footnotetext{\hspace{-1.5em}{\it 2020 Mathematics Subject Classification:} 
35Q40, 35L70 \\
{\it Key words and phrases:} Yang-Mills,  Dirac equation,
local well-posedness, temporal gauge}
\normalsize 
\setcounter{section}{0}

\section{Introduction and the main theorem}

\noindent 
Let $\mathcal{G}$ be the Lie group $SU(n,\mathbb{C})$ (the group of unitary matrices of determinant 1) and $g$ its Lie algebra $su(n,\mathbb{C})$ (the algebra of trace-free skew hermitian matrices) with Lie bracket $[X,Y] = XY-YX$ (the matrix commutator). 
For given  $A_{\alpha}: \mathbb{R}^{1+3} \rightarrow g $ we define the curvature $F=F[A]$ by
\begin{equation}
\label{curv}
 F_{\mu \nu} = \partial_{\mu} A_{\nu} - \partial_{\nu} A_{\mu} + [A_{\mu},A_{\nu}] \, , 
\end{equation}
where $\mu,\nu \in \{0,1,2,3\}$ and $D_{\mu} = \partial_{\mu} + [A_{\mu}, \cdot \,]$ .

Then the Yang-Mills system is given by
\begin{equation}
\label{0}
D^{\mu} F_{\mu \nu}  = 0
\end{equation}
in Minkowski space $\mathbb{R}^{1+3} = \mathbb{R}_t \times \mathbb{R}^3_x$ , with metric $diag(-1,1,1,1)$. Greek indices run over $\{0,1,2,3\}$, Latin indices over $\{1,2,3\}$, and the usual summation convention is used.  
We use the notation $\partial_{\mu} = \frac{\partial}{\partial x_{\mu}}$, where we write $(x^0,x^1,x^2,x^3)=(t,x^1,x^2,x^3)$ and also $\partial_0 = \partial_t$.

This system is coupled with a Dirac spinor field $\psi: \R^{1+3} \to \mathbb{C}^4$ . Let $T_a$ be the set of generators
 of $SU(n,\mathbb{C})$ and $A_{\mu} = A^a_{\mu} T_a$ , $F_{\mu \nu} = F^a_{\mu \nu} T_a$ , $[T^\lambda,T^b]_a =: f^{ab \lambda}$ .

For the following considerations and also for the physical background we refer to the monograph by Matthew D. Schwartz \cite{Sz} . We also refer to the pioneering work for the Yang-Mills, Higgs and spinor field equations by Y. Choquet-Bruhat and D.  Christodoulou \cite{CC} , and G. Schwarz and J. Sniatycki \cite{SS}.

The kinetic Lagrangian with $N$ Dirac fermions and the Yang-Mills Lagrangian are given by
$\mathcal{L} = \sum_{j=1}^N \bar{\psi}_j  (i \gamma^{\mu} \partial_{\mu}-m)\psi_j $ and $\mathcal{L}_{YM} = - \frac{1}{4} (F^a_{\mu \nu})^2$ , respectively. Here $\bar{\psi} = \psi^{\dagger} \gamma^0$ , where $\psi^{\dagger}$ is the complex conjugate transpose of $\psi$ .  

Here $\gamma^{\mu}$ are the (4x4) Dirac matrices given by
$ \, \, \gamma^0 = \left( \begin{array}{cc}
I & 0  \\ 
0 & -I  \end{array} \right)\, \,$ 
 , $\, \,  \gamma^j = \left( \begin{array}{cc}
0 & \sigma^j  \\
-\sigma^j & 0  \end{array} \right) \, \, $ , where  $\, \, \sigma^1 = \left( \begin{array}{cc}
0 & 1  \\
1 & 0  \end{array} \right)$ ,
$ \sigma^2 = \left( \begin{array}{cc}
0 & -i  \\
i & 0  \end{array} \right)$ ,
$ \sigma^3 = \left( \begin{array}{cc}
1 & 0  \\
0 & -1  \end{array} \right)$ .
Then we consider the following Lagrangian for the (minimally)  coupled system
\begin{align*}
 \mathcal{L}&=- \frac{1}{4} (F^a_{\mu \nu})^2 + \sum_{i,j=1}^N \bar{\psi}_i (\delta_{ij} i \gamma^{\mu} \partial_{\mu}+\gamma^{\mu}A^a_{\mu} T^a_{ij}     -m \delta_{ij})\psi_j \\
& =-\frac{1}{4}(\partial_{\mu} A^a_{\nu}-\partial_{\nu} A^a_{\mu} + f^{abc} A^b_{\mu} A^c_{\nu})^2 +\sum_{i,j=1}^N \bar{\psi}_i (\delta_{ij} i \gamma^{\mu} \partial_{\mu}+\gamma^{\mu}A^a_{\mu} T^a_{ij}     -m \delta_{ij})\psi_j \, . 
\end{align*}
Here $T^a_{ij} \in \mathbb{C}$ are the entries of the matrix $T^a$ . 

The corresponding equations of motion are given by the following coupled Yang-Mills-Dirac system (YMD)
\begin{align*} 
 \partial^{\mu} F^a_{\mu \nu} + f^{abc} A^b_{\mu} F^c_{\mu \nu} &= - \langle \psi_i,\gamma_0 \gamma_{\nu} T^a_{ij} \psi_j \rangle \\
(i \gamma^{\mu} \partial_{\mu} -m) \psi_i & = - A^a_{\mu} \gamma^{\mu} T^a_{ij} \psi_j \, .
\end{align*}
Using $D^{\mu} F_{\mu \nu} = \partial^{\mu} F_{\mu \nu} + [A^{\mu},F_{\mu \nu}]$ and 
$$ [A^{\mu},F_{\mu \nu}]_a = [A^{\lambda}_{\mu} T_{\lambda},F^b_{\mu \nu} T_b]_a = A^{\lambda}_{\mu} F^b_{\mu \nu} [T_{\lambda},T_b]_a = A^b_{\mu} F^b_{\mu \nu} f_{ab\lambda} $$
we obtain the following system which we intend to treat:
\begin{align}
\label{0.1}
D^{\mu} F_{\mu \nu} & = - \langle \psi^i,\alpha_{\nu} T^a_{ij} \psi^j \rangle T_a \, , \\
\label{0.2}
i \alpha^{\mu} \partial_{\mu} \psi_i & = -A^a_{\mu} \alpha^{\mu} T^a_{ij} \psi_j \, ,
\end{align}
if we choose $m=0$ just for simplicity and define the matrices $\alpha^{\mu} = \gamma^0 \gamma^{\mu}$ , so that $\, \,\alpha^0 = I_{4x4}$ and $ \alpha^j = \left( \begin{array}{cc}
0 & \sigma^j  \\
\sigma^j & 0  \end{array} \right)$.
$\alpha^{\mu}$ are hermitian matrices with $(\alpha^{\mu})^2 = I_{4x4}$ , $\alpha^j \alpha^k + \alpha^k \alpha^j = 0$ for $j \neq k$ .

Setting $\nu =0$ in (\ref{0.1}) we obtain the Gauss-law constraint
\begin{equation}
	\nonumber
	\partial^j F_{j 0} = -[A^j,F_{j0}] + \langle \psi^i,T^a_{ij} \psi^j \rangle T_a \,. 
\end{equation}

The system is gauge invariant. Given a sufficiently smooth function $U: {\mathbb R}^{1+3} \rightarrow \mathcal{G}$ we define the gauge transformation $T$ by $T A_0 = A_0'$ , 
$T(A_1,A_2,A_3) = (A_1',A_2',A_3')$ , $T\psi = \psi'$ , where
\begin{align*}
	A_{\alpha} & \longmapsto A_{\alpha}' = U A_{\alpha} U^{-1} - (\partial_{\alpha} U) U^{-1} \\
	\psi & \longmapsto \psi' = U \psi \, . 
\end{align*}

Following \cite{AFS1} and \cite{HO} in order to  rewrite the Dirac equation we define the projections
$$\Pi(\xi) := \half(I_{4x4} + \frac{\xi_j \alpha^j}{|\xi|})$$  and $\Pi_{\pm}(\xi):= \Pi(\pm \xi)$ , so that $\Pi_{\pm}(\xi)^2 = \Pi_{\pm}(\xi)$ , $\Pi_+(\xi) \Pi_-(\xi) =0 $ , $\Pi_+ (\xi) + \Pi_-(\xi) = I_{4x4}$ , $\Pi_{\pm}(\xi) = \Pi_{\mp}(-\xi) $ .
We obtain 
\begin{equation}
\label{2.6'}
\alpha^j \Pi(\xi) = \Pi(\pm \xi) \alpha^j + \frac{\xi_j}{|\xi|} I_{4x4} \, .
\end{equation}
Using the notation $\Pi_{\pm} = \Pi_{\pm}(\frac{\nabla}{i})$ we obtain
\begin{equation}
\label{2.8}
 -i\alpha^j \partial_j = |\nabla|\Pi_+ - |\nabla|\Pi_- \, , 
\end{equation}
where $|\nabla|$ has symbol $|\xi| $ . Moreover defining the modified Riesz transform by $R^j_{\pm} = \mp(\frac{\partial_j}{i|\nabla|}) $ with symbol $ \mp \frac{\xi_j}{|\xi|}$ and $R^0_{\pm} = -1$ the identity (\ref{2.6'}) implies
\begin{equation}
\label{2.7}
\alpha^j \Pi_{\pm} = (\alpha^j \Pi_{\pm})\Pi_{\pm} = \Pi_{\mp} \alpha^j \Pi_{\pm}- R^j_{\pm} \Pi_{\pm} \, , \, \alpha^0 \Pi_{\pm}= \Pi_{\pm} = \Pi_{\mp} \alpha^0 \Pi{\pm} - R^0_{\pm} \Pi_{\pm} \, .
\end{equation}
If we define $\psi_{i,\pm} = \Pi_{\pm} \psi_i$ we obtain by applying the projection $\Pi_{\pm}$ and (\ref{2.8}) the Dirac type equation in the form
\begin{equation}
\label{0.3}
(i \partial_t \pm |\nabla|)\psi_{i,\pm} = \Pi_{\pm}(A^a_{\mu} \alpha^{\mu} T^a_{ij} \psi^j) = : H_{i,\pm}(A,\psi) \, .
\end{equation}

The Yang-Mills equation (\ref{0.1}) may be written as
$$\square  A_{\nu} = \partial_{\nu} \partial^{\mu} A_{\mu}-[\partial^{\mu} A_{\mu},A_{\nu}] - [A_{\mu},\partial^{\mu} A_{\nu}] - [A^{\mu},F_{\mu \nu}] - \langle \psi_i,\alpha_{\nu} T^a_{ij} \psi^j \rangle T_a \, . $$

From now on we impose the temporal gauge
$$ A_0 = 0 \, . $$
This implies the wave equation
\begin{align}
\nonumber
\square  A_j &= \partial_j div\, A - [div\, A,A_j] - [A_i,\partial^i A_j] - [A^i,F_{ij}]- \langle \psi^i,\alpha_k T^a_{ik} \psi^k \rangle T_a \\
&= \partial_j div\, A - [div\, A,A_j] - 2[A_i,\partial^i A_j] + [A^i,\partial_j A_i] - [A^i,[A_i,A_j]]\\
& - \langle \psi^i,\alpha_k T^a_{ik} \psi^k \rangle T_a 
\label{0.4} 
\end{align}
and
$$
0 = \partial_t div \, A - [A_i,\partial_t A^i] - \langle \psi^i,\alpha_k T^a_{ik} \psi^k \rangle T_a \, . $$

Now we use the Hodge decomposition of $A=(A_1,A_2,A_3)$ into its divergence-free and curl-free parts:
$$A= A^{df} + A	^{cf} \, , $$
where 
$$PA := A^{df} = |\nabla|^{-2} \nabla \times(\nabla \times A) \,
\Leftrightarrow \,  A^{df}_j = R^k(R_j A_k-R_k A^j) $$
and 
$$A^{cf} = -|\nabla|^{-2} \nabla(div \, A) \quad \Leftrightarrow \quad A	^{cf}_j =-R_j R_k A^k \, . $$ 
Here $R_j = \frac{\partial_j}{|\nabla|}$ is the Riesz transform. \\
Then we obtain the following system:
\begin{align} 
	\label{6}
	\partial_t A^{cf} & = |\nabla|^{-2} \nabla [A_i,\partial_t A^i] + |\nabla|^{-2} \nabla \langle \psi^i,\alpha_k T^a_{ik} \psi^k \rangle T_a \, , \\
	\nonumber
	\square  A_j^{df} 
	&= -P [div\, A^{cf},A_j] - 2P[A_i,\partial^i A_j] + P[A^i,\partial_j A_i] - P[A^i,[A_i,A_j]] \\
	\label{7}
	& \quad - P \langle \psi^i,\alpha_k T^a_{ik} \psi^k \rangle T_a \\
	\label{8}
	(i \partial_t \pm |\nabla|)\psi_{i,\pm} & = \Pi_{\pm} (A^a_k \alpha^k T^a_{ij} \psi^j) \, .
\end{align}

We want to solve the system (\ref{6}),(\ref{7}),(\ref{8}) simultaneously for $A^{cf}$ , $A^{df}$ and $\psi_{\pm}$ .
So to pose the Cauchy problem for this system, we consider initial data for $(A^{df},A^{cf},\psi)$ at $t=0$:
\begin{equation}\label{Data}
\begin{split}
&A^{df}(0) = a_0^{df}, \, (\partial_t A^{df})(0) =  a_1^{df},
        \, A^{cf}(0) = a_0^{cf}\\
         & \, \psi_{i,\pm}(0) = \psi_{i,\pm}^0 =  \Pi_{\pm} \psi_0^i .
\end{split}
 \end{equation}

Let us make some historical remarks. As is well-known we may impose a gauge condition. Convenient gauges are the Coulomb gauge $\partial^j A_j=0$ , the Lorenz gauge $\partial^{\alpha}A_{\alpha} =0$ and the temporal gauge $A_0 =0$.  It is well-known that for the low regularity well-posedness problem for the Yang-Mills equation a null structure for some of the nonlinear terms plays a crucial role. This was first detected by Klainerman and Machedon \cite{KM}, who proved global well-posedness in the case of three space dimensions in temporal and in Coulomb gauge in energy space. The corresponding result in Lorenz gauge, where the Yang-Mills equation can be formulated as a system of nonlinear wave equations, was shown by Selberg and Tesfahun \cite{ST}, who discovered that also in this case some of the nonlinearities have a null structure.  Tesfahun \cite{Te} improved the local well-posedness result to data without finite energy, namely for $(A(0),(\partial_t A)(0) \in H^s \times H^{s-1}$ and $(F(0),(\partial_t F)(0) \in H^r \times H^{r-1}$ with $s > \frac{6}{7}$ and $r > -\frac{1}{14}$, by discovering an additional partial null structure. 
Local well-posedness in energy space was also shown by Oh \cite{O} using a new gauge, namely the Yang-Mills heat flow. He was also able to show that this solution can be globally extended \cite{O1}. Tao \cite{T1} showed local well-posedness for small data in $H^s \times H^{s-1}$ for $ s > \frac{3}{4}$ in temporal gauge.

The coupled Yang-Mills and Dirac system in Lorenz gauge was considered from the physical point of view by M. D. Schwartz \cite{Sz}. Local existence for smooth initial data, uniqueness in suitable gauges under appropriate conditions on the data and global existence for small and smooth data , i.e. $(A(0),(\partial_t A)(0),F(0),(\partial_t F)(0),\\ \psi(0)) \in H^s \times H^{s-1} \times H^{s-1} \times H^{s-2} \times H^s$ with $s \ge 2$  was proven by  Y. Choquet-Bruhat and D. Christodoulou \cite{CC}, and G. Schwarz and J. Sniatycki \cite{SS}. 

In \cite{P1} the author considered this problem in Lorenz gauge and obtained local well-posedness for $s > \frac{3}{4}$ , $r >-\frac{1}{8}$ and $l > \frac{3}{8}$ , where existence holds in $ A \in C^0([0,T],H^s) \cap C^1([0,T],H^{s-1}) \, , \, F \in C^0([0,T],H^r) \cap C^1([0,T],H^{r-1}) $, $\psi \in C^0([0,T],H^l)$ and (existence and) uniqueness in a certain subspace of Bourgain-Klainerman-Machedon type $X^{s,b}$ .  We relied on Selberg-Tesfahun \cite{ST} and Tesfahun's  result \cite{Te}, who detected   the null structure in most - unfortunately not all  - critical nonlinear terms.  We also made use of the methods used by Huh and Oh \cite{HO} for the Chern-Simons-Dirac equation.

We now study the Yang-Mills-Dirac system in temporal gauge for low regularity data,  which fulfill a smallness assumption, which reads as follows
$$ \|A(0)\|_{H^s} + \|(\partial_t A)(0)\|_{H^{s-1}} + \|\psi(0)\|_{H^l} < \epsilon $$
with a sufficiently small $\epsilon > 0$ , under the assumption $s> \frac{3}{4}$  and $l > \frac{1}{4}$ . We obtain a solution which satisfies $A \in C^0([0,1],H^s) \cap C^1([0,1],H^{s-1})$, $\psi \in C^0([0,1],H^l)$.
Uniqueness holds in spaces of Bourgain-Klainerman-Machedon type.
Thus the parameter $l$ can be weakened compared to the result in Lorenz gauge at the expense of a smallness assumption on the data.
The basis for our results is Tao's paper \cite{T1}, who considered the corresponding result for the Yang-Mills equation. We carry over his result to the more general Yang-Mills-Dirac equation. The result relies on the null structure of all the critical bilinear terms. We review this null structure which was partly detected already by Klainerman-Machedon \cite{KM1}  in the situation of the Lorenz gauge. The necessary estimates for those nonlinear terms, which contain no terms depending on $\psi$ ,  in spaces of $X^{s,b}$-type then reduce essentially to Tao's result \cite{T1}. One of these estimates is responsible for the small data assumption. Because these local well-posedness results can initially only be shown under the condition that the curl-free part $A^{cf}$ of $A$ (as defined below) vanishes for $t=0$ we have to show that this assumption can be removed by a suitable gauge transformation (Lemma \ref{Lemma}) which preserves the regularity of the solution. This uses an idea of Keel and Tao \cite{T1}. A proof for the Yang-Mills and Yang-Mills-Higgs case was given by \cite{P1}.\\[0.5em]

Our main theorem reads as follows:
\begin{theorem}
	\label{Theorem1.1}
	Let  $s > \frac{3}{4}$ , $l > \frac{1}{4}$ , $s \ge l \ge s-1$ , $2s-l > 1$ and $l-s \ge -\half$.  Let $a_0 \in H^s({\mathbb R}^3)$ , $a_1 \in H^{s-1}({\mathbb R}^3)$ , $\psi_0 \in H^l({\mathbb R}^3)$ be given  satisfying the Gauss law constraint $\partial^j a_j^1 = -\partial^j a_j^1+ \langle \psi^i_0,T^a_{ij} \psi^j_0 \rangle T_a$. Assume
	$$ \|a_0\|_{H^s} + \|a_1\|_{H^{s-1}} + \|\psi_0\|_{H^l}  \le \epsilon \, , $$
	where $\epsilon > 0$ is sufficiently small. Then the Yang-Mills-Dirac equations (\ref{0.1}) , (\ref{0.2}) in temporal gauge $A_0=0$ with initial conditions
	$$ A(0)=a_0 \, , \, (\partial_t A)(0) = a_1 \, , \, \psi(0)=\psi_0 \,,$$
	where $A=(A_1,A_2,A_3)$,
	has a unique local solution  $A= A_+^{df} + A_-^{df} +A^{cf}$ and $\phi = \phi_+ + \phi_-$ , where
	$$ A^{df}_{\pm} \in X^{s,\frac{3}{4}+}_{\pm}[0,1]  ,  A^{cf} \in X^{s+\frac{1}{4},\frac{1}{2}+}_{\tau=0}[0,1]  ,  \partial_t A^{cf} \in C^0([0,1],H^{s-1})  , 
	\psi_{\pm} \in  X^{l,\half+}_{\pm}[0,1] \, , $$
	where these spaces are defined below. This solution fulfills
	$$ A \in C^0([0,1],H^s({\mathbb R}^3)) \cap C^1([0,1],H^{s-1}({\mathbb R}^3)) \, , \, \psi \in C^0([0,1],H^l({\mathbb R}^3)) \, .$$
\end{theorem}

\begin{Def}
	\label{Def.1.2}
	The standard spaces $X^{s,b}_{\pm}$ of Bourgain-Klainerman-Machedon type belonging to the half waves are the completion of the Schwarz space  $\mathcal{S}({\mathbb R}^4)$ with respect to the norm
	$$ \|u\|_{X^{s,b}_{\pm}} = \| \langle \xi \rangle^s \langle  \tau \mp |\xi| \rangle^b \widehat{u}(\tau,\xi) \|_{L^2_{\tau \xi}} \, . $$ 
	Similarly we define the wave-Sobolev spaces $X^{s,b}_{|\tau|=|\xi|}$ with norm
	$$ \|u\|_{X^{s,b}_{|\tau|=|\xi|}} = \| \langle \xi \rangle^s \langle  |\tau| - |\xi| \rangle^b \widehat{u}(\tau,\xi) \|_{L^2_{\tau \xi}}  $$ and also $X^{s,b}_{\tau =0}$ with norm 
	$$\|u\|_{X^{s,b}_{\tau=0}} = \| \langle \xi \rangle^s \langle  \tau  \rangle^b \widehat{u}(\tau,\xi) \|_{L^2_{\tau \xi}} \, .$$
	We also define $X^{s,b}_{\pm}[0,T]$ as the space of the restrictions of functions in $X^{s,b}_{\pm}$ to $[0,T] \times \mathbb{R}^3$ and similarly $X^{s,b}_{|\tau| = |\xi|}[0,T]$ and $X^{s,b}_{\tau =0}[0,T]$. We frequently use the estimates $\|u\|_{X^{s,b}_{\pm}} \le \|u\|_{X^{s,b}_{|\tau|=|\xi|}}$ for $b \le 0$ and the reverse estimate for $b \ge 0$. 
\end{Def}

We recall the fact that
$$
X^{s,b}_\pm [0,T]  \hookrightarrow C^0([-T,T];H^s) \quad \text{for} \ b > \half.
$$
We use the following notation:
let $\langle \nabla \rangle^{\alpha}$ , $D^{\alpha} = |\nabla|^{\alpha}$ and $D_{-}^{\alpha}$ be the multipliers with symbols  $
\langle\xi \rangle^\alpha$ , $
\abs{\xi}^\alpha$ and $ ||\tau|-|\xi||^\alpha$ ,
respectively, where $\langle \cdot \rangle = (1+|\cdot|^2)^{\half}$ . Finally $a{\pm}$ and $a{\pm}{\pm}$ is short for  $a{\pm}\epsilon$ and $a{\pm}2\epsilon$ for a sufficiently small $\epsilon >0$ .

\section{Preliminaries}
  The following product estimates for wave-Sobolev spaces were proven in \cite{AFS}.
\begin{prop}
\label{Prop.1.2'} 
 For $s_0,s_1,s_2,b_0,b_1,b_2 \in {\mathbb R}$ and $u,v \in   {\mathcal S} ({\mathbb R}^{3+1})$ the estimate
$$\|uv\|_{H^{-s_0,-b_0}} \lesssim \|u\|_{H^{s_1,b_1}} \|v\|_{H^{s_2,b_2}} $$ 
holds, provided the following conditions are satisfied:
\begin{align*}
\nonumber
& b_0 + b_1 + b_2 > \frac{1}{2} \, ,
& b_0 + b_1 \ge 0 \, ,\quad \qquad  
& b_0 + b_2 \ge 0 \, ,
& b_1 + b_2 \ge 0
\end{align*}
\begin{align*}
\nonumber
&s_0+s_1+s_2 > 2 -(b_0+b_1+b_2) \\
\nonumber
&s_0+s_1+s_2 > \frac{3}{2} -\min(b_0+b_1,b_0+b_2,b_1+b_2) \\
\nonumber
&s_0+s_1+s_2 > 1 - \min(b_0,b_1,b_2) \\
\nonumber
&s_0+s_1+s_2 > 1 \\
 &(s_0 + b_0) +2s_1 + 2s_2 > \frac{3}{2} \\
\nonumber
&2s_0+(s_1+b_1)+2s_2 > \frac{3}{2} \\
\nonumber
&2s_0+2s_1+(s_2+b_2) > \frac{3}{2}
\end{align*}
\begin{align*}
\nonumber
&s_1 + s_2 \ge \max(0,-b_0) \, ,\quad
\nonumber
s_0 + s_2 \ge \max(0,-b_1) \, ,\quad
\nonumber
s_0 + s_1 \ge \max(0,-b_2)   \, .
\end{align*}
\end{prop}

\begin{prop}[Null form estimates, \cite{ST} ]  
\label{Prop.1.2}
Let $\sigma_0,\sigma_1,\sigma_2,\beta_0,\beta_1,\beta_2 \in \R$. Assume that
\begin{equation*}
\left\{
\begin{aligned} 
 & 0 \le \beta_0 < \frac12 < \beta_1,\beta_2 < 1,
  \\
 & \sum \sigma_i + \beta_0 > \frac32 - (\beta_0 + \sigma_1 + \sigma_2),
  \\
 & \sum \sigma_i > \frac32 - (\sigma_0 + \beta_1 + \sigma_2),
  \\
&  \sum \sigma_i > \frac32 - (\sigma_0 + \sigma_1 + \beta_2),
  \\
  &\sum \sigma_i + \beta_0 \ge 1,
  \\
 & \min(\sigma_0 + \sigma_1, \sigma_0 + \sigma_2, \beta_0 + \sigma_1 + \sigma_2) \ge 0,
\end{aligned}
\right.
\end{equation*} 
and that the last two inequalities are not both equalities. Let
\begin{align}
\nonumber
&{\mathcal F}(B_{\pm_1,\pm_2} (\psi_{1_{\pm_1}}, \psi_{2_{\pm_2}}))(\tau_0,\xi_0) \\
\label{2}
& := \int_{\tau_1+\tau_2= \tau_0\, \xi_1+\xi_2=\xi_0} |\angle(\pm_1 \xi_1,\pm_2 \xi_2)|  \widehat{\psi_{1_{\pm_1}}}(\tau_1,\xi_1)  \widehat{\psi_{2_{\pm_2}}}(\tau_2,\xi_2) d\tau_1 d\xi_1 \, . 
\end{align}
Then we have the null form estimate
$$
  \norm{B_{(\pm_1 \xi_1,\pm_2 \xi_2)}(u,v)}_{H^{-\sigma_0,-\beta_0}}
  \lesssim
  \norm{u}_{X^{\sigma_1,\beta_1}_{\pm_1}} \norm{v}_{X^{\sigma_2,\beta_2}_{\pm 2}}\, .
$$
\end{prop}

The following multiplication law is well-known:
\begin{prop} {\bf (Sobolev multiplication law)}
\label{SML}
Let $s_0,s_1,s_2 \in \R$ . Assume
$s_0+s_1+s_2 > \frac{3}{2}$ , $s_0+s_1 \ge 0$ ,  $s_0+s_2 \ge 0$ , $s_1+s_2 \ge 0$. Then the following product estimate holds:
$$ \|uv\|_{H^{-s_0}} \lesssim \|u\|_{H^{s_1}} \|v\|_{H^{s_2}} \, .$$
\end{prop}

\begin{prop}
	\label{Prop.2}
	\begin{enumerate}
		\item For $2 < q \le \infty $ , $ 2 \le r < \infty$ , $ \frac{1}{q} = \frac{1}{2}-\frac{1}{r}$ , $ \mu = 3(\frac{1}{2}-\frac{1}{r})-\frac{1}{q}$ the following estimate holds
		\begin{equation}
			\label{15}
			\|u\|_{L^q_t L^r_x} \lesssim \|u\|_{X^{\mu,\frac{1}{2}+}_{|\tau|=|\xi|}} \, . 
		\end{equation}
		\item For $k \ge 0$ , $ p < \infty$ and $ \frac{1}{4} \ge \frac{1}{p} \ge \frac{1}{4} - \frac{k}{3}$ the following estimate holds:
		\begin{equation}
			\label{Tao}
			\|u\|_{ L^p_x L^2_t} \lesssim \|u\|_{X^{k+\frac{1}{4},\frac{1}{2}+}_{|\tau|=|\xi|}} \, . 
		\end{equation}
	\end{enumerate}
\end{prop}
\begin{proof}
	(\ref{15}) is the Strichartz type estimate, which can be found for e.g. in \cite{GV}, Prop. 2.1, combined with the transfer principle (cf. e.g. \cite{KS}, Prop. 3.5).
	
	Concerning (\ref{Tao}) we refer to \cite{T}, Prop. 4.1, or use \cite{KMBT}, Thm. B.2:
	$$ \|\mathcal{F}_t u \|_{L^2_{\tau} L_x^4} \lesssim \|u_0\|_{\dot{H}^{\frac{1}{4}}} \, , $$
	if $u=e^{it |\nabla|} u_0$ and $\mathcal{F}_t$ denotes the Fourier transform with respect to time. This immediately implies by Plancherel, Minkowski's inequality and Sobolev's embedding theorem
	$$\|u\|_{L^p_x L^2_t} = \|\mathcal{F}_t u \|_{L^p_x L^2_\tau} \le \|\mathcal{F}_t u \|_{L^2_{\tau} L^p_x} \lesssim \|\mathcal{F}_t u \|_{L^2_{\tau} H^{k,4}_x} \lesssim \|u_0\|_{H^{k+\frac{1}{4}}} \, . $$
	The transfer principle implies (\ref{Tao}).
\end{proof}

\section{Preliminary local well-posedness}
Defining
\begin{align*}
	A^{df}_{\pm} = \frac{1}{2}(A^{df} \mp i \langle \nabla \rangle^{-1} \partial_t A^{df}) & \Longleftrightarrow A^{df} = A^{df}_+ + A_-^{df} \, , \, \partial_t A^{df} = i \langle \nabla \rangle(A^{df}_+ - A^{df}_-)
\end{align*}
we can rewrite (\ref{7}) as
\begin{align}
	\label{7'}
	(i \partial_t \pm \langle \nabla \rangle)A_{j,\pm}^{df} & = \mp 2^{-1} \langle \nabla \rangle^{-1} ( R.H.S. \, of \, (\ref{7}) - A^{df}_j) \, .
\end{align}
with initial data
\begin{align}
	\label{1.15*'}
	A^{df}_\pm(0) & = \frac{1}{2}(A^{df}(0) \mp i^{-1} \langle \nabla \rangle^{-1} (\partial_t A^{df})(0)) \, .
\end{align}

	We now state and prove a preliminary local well-posedness result for (\ref{6}),(\ref{7}), (\ref{8}), for which it is essential to have data for $A$ with vanishing curl-free part. 

\begin{prop}
	\label{Prop}
	Assume $s>\frac{3}{4}$ , $l > \frac{1}{4}$ ,  $s \ge l \ge s-1$ , $2s-l > 1$ and $l-s \ge -\half$.
	Let
	$a_0^{df}  \in H^s$ , $a_1^{df} \in H^{s-1}$ , $\psi_0 \in H^l$ be given  satisfying the Gauss law constraint $\partial^j a_j^1 = -\partial^j a_j^1+ \langle \psi^i_0,T^a_{ij} \psi^j_0 \rangle T_a$ (necessary by (\ref{0.1}) with $\nu=0$) with
	$$ \|a_0^{df}\|_{H^s} + \|a_1^{df}\|_{H^{s-1}} + \|\psi_0 \|_{H^l} \le \epsilon_0$$
	where $\epsilon_0 >0$ is sufficiently small. Then the system (\ref{6}),(\ref{7})(\ref{8}) with initial conditions
	$$ A^{df}(0)=a_0^{df} \, , \, (\partial_t A^{df})(0) = {a_1}^{df} \, , \, A^{cf}(0) = 0 \, , \, \psi(0)= \psi_0$$
	has a unique local solution 
	$$ A= A^{df}_+ + A^{df}_- + A^{cf} \, , \, \psi = \psi_+ + \psi_- \, , $$
	where 
	$$  A^{df}_{\pm} \in X^{s,\frac{3}{4}+}_{\pm}[0,1] \, , \, A^{cf} \in X^{s+\frac{1}{4},\frac{1}{2}+}_{\tau=0}[0,1] \, , \, \partial_t A^{cf} \in C^0([0,1],H^{s-1}) \, , \, \psi_{\pm} \in X^{l,\half+}_{\pm}[0,1] .$$
	Uniqueness holds (of course) for not necessarily vanishing initial data $A^{cf}(0) = a^{cf}$. The solution satisfies
	$$ A \in C^0([0,1],H^s) \cap C^1([0,1],H^{s-1}) \, , \, \psi \in C^0([0,1],H^l) \, .$$
\end{prop}

We want to use a contraction argument for  $A_{\pm}^{df} \in X_{\pm}^{s,\frac{3}{4}+\epsilon}[0,1] \, , \, A^{cf} \in X^{s+\frac{1}{4},\frac{1}{2}+\epsilon}_{\tau=0}$ $[0,1]$ , $\partial_t A^{cf} \in C^0([0,1],H^{s-1}$ ) and $\psi_{\pm} \in X_{\pm}^{l,\half+ \epsilon}[0,1]$ .
Provided that our small data assumption holds this can be reduced  by well-known arguments to suitable multilinear estimates of the right hand sides of these equations. 
For (\ref{7'}) e.g. we make use of the following well-known estimate:
$$
\|A^{df}_{\pm}\|_{X^{s,b}_{\pm}[0,1]} \lesssim \|A^{df}_{\pm}(0)\|_{H^s} +  \| R.H.S. \, of \, (\ref{7'}) \|_{X^{s,b-1}_{\pm}[0,1]} \, , $$
which holds for $s\in{\mathbb R}$  , $\frac{1}{2} < b \le 1$ .
For (\ref{6}) we make use of the estimate:
$$
\|A^{cf}\|_{X^{s+\frac{1}{4},b}_{\tau=0}[0,1]} \lesssim   \| R.H.S. \, of \, (\ref{6}) \|_{X^{s+\frac{1}{4},b-1}_{\pm}[0,1]} \, . $$
Here it is essential, that no term containing $A^{cf}(0)$ appears on the right hand, because we do not want to assume $A^{cf}(0) \in H^{s+\frac{1}{4}}$ . Therefore in a first step we assume $A^{cf}(0)=0$ , which we later remove by an application of a suitable gauge transform. We are forced to admit the same parameter $b$ on both sides of the latter inequality at one point, which prevents a large data result.

We now show that all the critical terms in (\ref{6}), (\ref{7}) and (\ref{8}), namely the quadratic terms which contain only $A^{df}$ or $\psi_{\pm}$ have null structure. Those quadratic terms which contain $A^{cf}$ are less critical, because $A^{cf}$ is shown to be more regular than $A^{df}$, and the cubic terms are also less critical, because they contain no derivatives.

What we have to prove are estimates for the right hand sides of (\ref{6}), (\ref{7'}) and (\ref{8}).

First we consider  the terms which do not contain $A^{cf}$ .

For the right hand side of (\ref{7'}) we have to prove the following estimates:
\begin{align}
	\label{7.1}
	\|P[A_{i,\pm_1}^{df},\partial^i A^{df}_{j,\pm_2}]\|_{H^{s-1,-\frac{1}{4}++}} & \lesssim \|A_{i,\pm_1}^{df}\|_{X^{s,\frac{3}{4}+}_{\pm_1}} \|A_{j,\pm_2}^{df}\|_{X^{s-1,\frac{3}{4}+}_{\pm_2}} \\
	\label{7.2}
		\|P[A_{\pm_1}^{df,i},\partial^j A^{df}_{i,\pm_2}]\|_{H^{s-1,-\frac{1}{4}++}} & \lesssim \|A_{i,\pm_1}^{df}\|_{X^{s,\frac{3}{4}+}_{\pm_1}} \|A_{i,\pm_2}^{df}\|_{X^{s-1,\frac{3}{4}+}_{\pm_2}} \\
		\label{7.3}
		\|P \langle \psi_{1,\pm_1},\alpha_j \psi_{2,\pm_2} \rangle \|_{X^{s-1,-\frac{1}{4}++}_{\pm_0}} & \lesssim \|\psi_{1,\pm_1}\|_{X^{l,\half+}_{\pm_1}}\|\psi_{2,\pm_2}\|_{X^{l,\half+}_{\pm_2}}
\end{align}
Concerning the right side of (\ref{8}), ignoring the irrelevant term $T^a_{ij}$ , we have to prove
\begin{equation}
	\label{8.1}
	\|\Pi_{\pm_0} (A_{k,df}^{\pm} \alpha^k \psi) \|_{X^{l,-\half++}_{\pm_1}} \lesssim \|\psi\|_{X^{l,\half+}_{\pm_1}} \|A_{k,df}^{\pm} \|_{X^{s,\frac{3}{4}+}_{\pm}} \, . 
\end{equation}
 In order to control $A^{cf}$ in (\ref{6}) we need
\begin{align}
	\label{16}
	\| |\nabla|^{-1} (\phi_1 \partial_t \phi_2)\|_{X^{s+\frac{1}{4},-\frac{1}{2}+\epsilon+}_{\tau=0}} &\lesssim \|\phi_1\|_{X^{s,\frac{3}{4}+\epsilon}_{|\tau|=|\xi|}} \|\phi_2\|_{X^{s,\frac{3}{4}+\epsilon}_{|\tau|=|\xi|}} \\
	\label{17}
	\| |\nabla|^{-1} (\phi_1 \partial_t \phi_2)\|_{X^{s+\frac{1}{4},-\frac{1}{2}+2\epsilon-}_{\tau=0}} &\lesssim \|\phi_1\|_{X^{s+\frac{1}{4},\frac{1}{2}+\epsilon}_{\tau=0}} \|\phi_2\|_{X^{s+\frac{1}{4},\frac{1}{2}+\epsilon}_{\tau=0}} \\
	\label{18}
	\| |\nabla|^{-1} (\phi_1 \partial_t \phi_2)\|_{X^{s+\frac{1}{4},-\frac{1}{2}+\epsilon}_{\tau=0}} &+ \| |\nabla|^{-1} (\phi_2 \partial_t \phi_1)\|_{X^{s+\frac{1}{4},-\frac{1}{2}+\epsilon}_{\tau=0}} \\ \nonumber
	&\lesssim \|\phi_1\|_{X^{s+\frac{1}{4},\frac{1}{2}+\epsilon}_{\tau=0}} \|\phi_2\|_{X^{s,\frac{3}{4}+\epsilon}_{|\tau|=|\xi|}} \, , \\
	\label{17a}
	\| |\nabla|^{-1} (\psi_1 \psi_2)\|_{X^{s+\frac{1}{4},-\half+\epsilon+}_{\tau =0}} & \lesssim \|\psi_1\|_{X^{l,\half+\epsilon}_{|\tau|=|\xi|}}\|\psi_2\|_{X^{l,\half+\epsilon}_{|\tau|=|\xi|}} \, .
\end{align} 
In order to control $\partial_t A^{cf}$  we need
\begin{align}
	\label{19}
	\| |\nabla|^{-1} (A_1 \partial_t A_2)\|_{C^0(H^{s-1})} \lesssim &
	(\|A_1^{cf}\|_{X^{s+\frac{1}{4},\frac{1}{2}+}_{\tau=0}} + \sum_{\pm} 
	\|A^{df}_{1\pm}\|_{X^{s,\frac{1}{2}+}_{\pm}})\\
	\nonumber
	&(\|\partial_t 
	A^{cf}_2\|_{C^0(H^{s-1})} +  \sum_{\pm} \|A^{df}_{2\pm}\|_{X^{s,\frac{1}{2}+}_{\pm}}) \, .
\end{align} 
and
\begin{equation}
	\label{19a}
	\| |\nabla|^{-2} \nabla \langle \psi_i, T^a_{ij} \psi_j \rangle T_a \|_{C^0(H^{s-1})} \lesssim \sum_{\pm_1} \|\psi_{i,\pm}\|_{X^{l,\half+}_{\pm_1}} \sum_{\pm_2} \|\psi_{j,\pm}\|_{X^{l,\half+}_{\pm_2}} \, . 
\end{equation}                        
Concerning (\ref{0.2}) it remains to consider
the terms, which contain a factor  $A^{cf}$. We need
\begin{equation}
	\label{29}
	\| \nabla A^{cf} A^{df} \|_{X^{s-1,-\frac{1}{4}+2\epsilon}_{|\tau|=|\xi|}} +
	\| A^{cf} \nabla A^{df} \|_{X^{s-1,-\frac{1}{4}+2\epsilon}_{|\tau|=|\xi|}} \lesssim \|A^{cf}\|_{X^{s+\frac{1}{4},\frac{1}{2}+\epsilon}_{\tau =0}}  \|A^{df}\|_{X^{s,\frac{3}{4}+\epsilon}_{|\tau|=|\xi|}} 
\end{equation}
and
\begin{equation}
	\label{30}
	\| \nabla A^{cf} A^{cf} \|_{X^{s-1,-\frac{1}{4}+2\epsilon}_{|\tau|=|\xi|}}  \lesssim \|A^{cf}\|_{X^{s+\frac{1}{4},\frac{1}{2}+\epsilon}_{\tau =0}}^2  \, .
\end{equation}
All the cubic terms are estimated by
\begin{equation}
	\label{31}
	\| A_1 A_2 A_3 \|_{X^{s-1,-\frac{1}{4}+2\epsilon}_{|\tau|=|\xi|}} \lesssim \prod_{i=1}^3 \min(\|A_i\|_{X^{s,\frac{3}{4}+\epsilon}_{|\tau|=|\xi|}},\|A_i\|_{X^{s+\frac{1}{4},\frac{1}{2}+\epsilon}_{\tau =0}} ) \, .
\end{equation}
Concerning (\ref{0.3}) it remains to prove the following estimate:
\begin{equation}
	\label{43}
	\|A^{cf} \psi\|_{X^{l,-\half++}_{|\tau|=|\xi|}} \lesssim \|A^{cf}\|_{X^{s+\frac{1}{4},\half+}_{\tau =0}} \|\psi\|_{X^{l,\half+}_{|\tau|=|\xi|}} \, .
\end{equation}

\begin{proof}[Proof of (\ref{7.1})] 
	We conclude
\begin{align}
	\nonumber
	&[A^{df}_i,\partial^i A^{df}] = [R^k(R_i A_k - R_k A_i),\partial^i A^{df}] \\
	\nonumber
	&= \frac{1}{2} \big([R^k(R_i A_k - R_k A_i),\partial^i A^{df}] + [R^i(R_k A_i - R_i A_k),\partial^k A^{df}]\big) \\
	\nonumber
	&=\frac{1}{2} \big([R^k(R_i A_k - R_k A_i),\partial^i A^{df}] - [R^i(R_i A_k - R_k A_i),\partial^k A^{df}]\big) \\
	\label{50}
	&= \frac{1}{2} Q^{ik} [ |\nabla|^{-1}(R_i A_k - R_k A_i),A^{df}] \,,
\end{align}
where
$$ Q_{ij}[u,v]  := [\partial_i u,\partial_jv] - [\partial_j u,\partial_i v] = Q_{ij}(u,v) + Q_{ji}(v,u) $$
with the standard null form
$$ Q_{ij}(u,v) := \partial_i u \partial_j v - \partial_j u \partial_i v \, . $$ 
Thus, ignoring $P$, which is a bounded operator, we obtain
\begin{equation}
	\label{N2}
	P[A_i^{df},\partial^i A^{df}] \sim \sum Q_{ik}[|\nabla|^{-1} A^{df},A^{df}] \, .
\end{equation}
It is well-known (cf. e.g. \cite{ST}) that the bilinear form $Q^{jk}_{\pm_{1}, \pm_{2}}$ ,  defined by
\begin{align*}
	& Q^{jk}_{\pm_{1}, \pm_{2}} (\phi_{1_{\pm_{1}}}, \phi_{2_{ \pm_{2}}}) 
	:= R^j_{\pm_{1}} \phi_{1_{\pm_{1}}} R^k_{\pm_{2}} \phi_{2_{ \pm_{2}}} - R^k_{\pm_{2}} \phi_{1_{ \pm_{1}}} R^j_{\pm_{1}} \phi_{1_{ \pm_{2}}} \, ,
\end{align*}
similarly to the standard null form $Q_{jk}$ , which is defined by replacing the modified Riesz transforms $R^k_{\pm}$ by $\partial^k$, fulfills the following estimate:
$$  Q^{jk}_{\pm_{1}, \pm_{2}} (\phi_{1_{ \pm_{1}}}, \phi_{2_{ \pm_{2}}}) \precsim B_{\pm_1,\pm_2}(\psi_{1_{\pm_1}},\psi_{2_{\pm_2}} )\, . $$
Let  $u \precsim v$ be defined by $|\widehat{u}| \lesssim |\widehat{v}|$ . 
We have to prove
$$ \|B_{\pm_1,\pm_2}(u,v) \|_{H^{s-1,-\frac{1}{4}+}} \lesssim \|u\|_{X^{s,\frac{3}{4}+}_{\pm_1}} \|v\|_{X^{s-1,\frac{3}{4}+}_{\pm_2}} \,. $$
This is implied by Prop. \ref{Prop.1.2} with parameters $\sigma_0=1-s$ , $\sigma_1 =s$ , $\sigma_2=s-1$ , $\beta_0=\frac{1}{4}-$ , $\beta_1=\beta_2=\frac{3}{4}+$ , provided $s >\frac{3}{4}$ .
\end{proof}

\begin{proof}[Proof of (\ref{7.2})]
\begin{align*}
(P(A_i^{df}  \nabla A_i^{df}))_j  & = R^k(R_j(A_i^{df} \partial_k A_i^{df}) - R_k( A_i^{df} \partial_j  A_i^{df})) \\
	& = |\nabla|^{-2} \partial^k(\partial_j ( A_i^{df} \partial_k  A_i^{df}) - \partial_k ( A_i^{df} \partial_j  A_i^{df}))) \\
	& = |\nabla|^{-2} \partial^k(\partial_j  A_i^{df} \partial_k  A_i^{df} - \partial_k  A_i^{df} \partial_j  A_i^{df}) \\
	& = |\nabla|^{-2} \partial^k Q_{jk}( A_i^{df}, A_i^{df})
\end{align*}
so that
\begin{equation}
	\label{N3}
	P[ A_i^{df},\nabla  A_i^{df}] \sim \sum |\nabla|^{-1} Q_{jk} [ A^{df}, A^{df}] \, .
\end{equation}

We have to prove
$$\| Q_{jk}(u,v) \|_{H^{s-2,-\frac{1}{4}+}} \lesssim \|u\|_{H^{s,\frac{3}{4}+}} \|v\|_{H^{s,\frac{3}{4}+}} \, . $$
We use the estimate 
$$ Q_{jk}(u,v) \precsim D^{\half} D_-^{\half} (D^{\half} u D^{\half} v) + D^{\half}(D^{\half} D_-^{\half} u D^{\half} v) + D^{\half}(D^{\half} u D^{\half} D_-^{\half} v) \, , $$
which was proven by \cite{KM2}, Prop. 1.
This reduces the proof to the estimates
\begin{align*}
	\|uv\|_{H^{s-\frac{3}{2},\frac{1}{4}+}} & \lesssim \|u\|_{H^{s-\half,\frac{3}{4}+}} \|v\|_{H^{s-\half,\frac{3}{4}+}} \, , \\
		\|uv\|_{H^{s-\frac{3}{2},-\frac{1}{4}+}} & \lesssim \|u\|_{H^{s-\half,\frac{1}{4}+}} \|v\|_{H^{s-\half,\frac{3}{4}+}} \, .	
\end{align*}
Both are implied by Prop. \ref{Prop.1.2'} with parameters $s_0= \frac{3}{2}-s$ , $s_1=s_2= s-\half$ , and $b_0=-\frac{1}{4}-$ , $b_1=b_2=\frac{3}{4}+$ for the first one and $b_0= \frac{1}{4}-$ , $b_1=\frac{1}{4}+$ , $b_2=\frac{3}{4}+$ for the second one, which both require the assumption $ s >\frac {3}{4}$ . 
\end{proof}

\begin{proof}[Proof of (\ref{7.3})] 
	Using the definition of $P$ and ignoring the irrelevant term $T^a$ we have to prove
	\begin{equation}
		\nonumber
		\| R^k(R_j \langle \psi_1,\alpha_k \psi_2 \rangle - R_k \langle \psi_1,\alpha_j \psi_2 \rangle) \|_{X_{\pm_0}^{s-1,-\frac{1}{4}++}} \lesssim \|\psi_1\|_{X^{l,\half+}_{\pm_1}} \|\psi_2\|_{X^{l,\half+}_{\pm_2}} \, .
	\end{equation}

	We obtain by (\ref{2.7}) :
	\begin{align*}
		&   R^k(R_j \langle \psi_1,\alpha_k \psi_2 \rangle - R_k \langle \psi_1,\alpha_j \psi_2 \rangle)\\
		&=  \sum_{\pm_1,\pm_2}R^k (R_j \langle \psi_{1_{\pm_1}}, \alpha_k \Pi_{\pm_2} \psi_{2_{\pm_2}} \rangle - R_k \langle \psi_{1_{\pm_1}}, \alpha_j \Pi_{\pm_2} \psi_{2_{\pm_2}} \rangle) \\
		& = \sum_{\pm_1,\pm_2} R^k (R_j \langle \psi_{1_{\pm_1}}, \Pi_{\mp_2}(\alpha_k  \psi_{2_{\pm_2}}) \rangle - R_k \langle \psi_{1_{\pm_1}}, \Pi_{\mp_2}(\alpha_j  \psi_{2_{\pm_2}}) \rangle)\\
		& \hspace{1em} -  \sum_{\pm_1,\pm_2} R^k(R_j \langle \psi_{1_{\pm_1}}, R^k_{\pm_2}  \psi_{2_{\pm_2}} \rangle - R_k \langle \psi_{1_{\pm_1}}, R^j_{\pm_2}  \psi_{2_{\pm_2}} \rangle)\\
		&= I + II \, .
	\end{align*}
	Both terms are null forms.
	
	Concerning I we consider each term separately and remark that at this point $R_k$ and $R_j$ are irrelevant. We obtain
	\begin{align}
		\label{30a}
		&{\mathcal F}(\langle \Pi_{\pm_1} \psi_{1_{\pm_1}},\Pi_{\mp_2}\alpha_k \psi_{2_{\pm_2}}\rangle) (\tau_0,\xi_0) \\
		\nonumber
		&= \int_{\tau_1+\tau_2=\tau_0 \, , \, \xi_1 + \xi_2= \xi_0} \langle \Pi(\pm_1 \xi_1) \widehat{\psi_{1_{\pm_1}}} (\tau_1,\xi_1),\Pi(\mp_2 \xi_2)\alpha_k \Pi(\pm_2 \xi_2) \widehat{\psi_{2_{\pm_2}}}(\tau_2,\xi_2)\rangle d\tau_1 d \xi_1 \\ \nonumber
		&= \int_{\tau_1+\tau_2=\tau_0 \, , \, \xi_1 + \xi_2= \xi_0} \langle  \widehat{\psi_{1_{\pm_1}}} (\tau_1,\xi_1),\Pi(\pm_1 \xi_1)\Pi(\mp_2 \xi_2)\alpha_k \Pi(\pm_2 \xi_2) \widehat{\psi_{2_{\pm_2}}}(\tau_2,\xi_2)\rangle d\tau_1 d \xi_1  .
	\end{align}
	Now we use the estimate ("spinorial null structure")
	$$ | \Pi(\pm \xi_1) \Pi(\mp \xi_2) z| \lesssim |z| \angle (\pm \xi_1,\pm \xi_2) $$ 
	proven by \cite{AFS}, Lemma 2. This implies
	$$ I \lesssim B_{\pm_1,\pm_2} (\psi_{1_{\pm_1}}, \psi_{2_{\pm_2}}) \, , $$
	where $ B_{\pm_1,\pm_2}$ is defined by (\ref{2}).
	
	We have to prove
	\begin{equation}
		\label{3}
		\| B_{\pm_1,\pm_2} (\psi_{1_{\pm_1}}, \psi_{2_{\pm_2}}) \|_{X_{\pm_0}^{s-1,-\frac{1}{4}++}} \lesssim \|\psi_{1_{\pm_1}} \|_{X^{l,\half+}_{\pm_1}} \|\psi_{2_{\pm_2}} \|_{X^{l,\half+}_{\pm_2}} \, .
	\end{equation}
We apply Prop. \ref{Prop.1.2} with parameters $\sigma_0=1-s$ , $\sigma_1=\sigma_2=l$ , $\beta_0=\frac{1}{4}-$ , $\beta_1=\beta_2=\half+$ . This requires $2l-s > - \frac{1}{4}$ , $4l-s >0$ and $3l-2s > -1$ , which follows from our assumptions.

	Next we obtain
	$$ II \precsim  \sum_{\pm_1,\pm_2} (R_{j,\pm 0} \langle \psi_{1_{\pm_1}}, R^k_{\pm_2}  \psi_{2_{\pm_2}} \rangle - R_{k,\pm 0} \langle \psi_{1_{\pm_1}}, R^j_{\pm_2}  \psi_{2_{\pm_2}} \rangle) \, . $$
	By duality we have to prove
	\begin{align*}
		&\left|\int\left(\langle \psi_{1_{\pm_1}},R^k_{\pm_2} \psi_{2_{\pm_2}} \rangle R^j_{\pm_0} \psi_{0_{\pm_0}} - \langle \psi_{1_{\pm_1}},R^j_{\pm_2} \psi_{2_{\pm_2}} \rangle R^k_{\pm_0} \psi_{0_{\pm_0}} \right) dx dt \right| \\
		& \hspace{1em}
		\lesssim \|\psi_{1_{\pm_1}}\|_{X^{l,\half+}_{\pm_1}} \|\psi_{2_{\pm_2}}\|_{X^{l,\half+}_{\pm_2}} \|\psi_{0_{\pm_0}}\|_{X^{-s+1,\frac{3}{4}--}_{\pm_0}} \, .
	\end{align*}
	We remark that the left hand side possesses a $Q^{jk}$-type null form between $\psi_{2_{\pm_2}}$ and $\psi_{0_{\pm_0}}$ .	We have to prove
	\begin{equation}
		\label{4}
		\|B_{\pm_2,\pm_0}(\psi_{2_{\pm_2}},\psi_{0_{\pm_0}}) \|_{X^{-l,-\half-}_{\pm_1}} \lesssim  \|\psi_{2_{\pm_2}}\|_{X^{l,\half+}_{\pm_2}} \|\psi_{0_{\pm_0}}\|_{X^{-s+1,\frac{3}{4}--}_{\pm_0}} \, .
	\end{equation}                           
                         We apply Prop. \ref{Prop.1.2} with $\sigma_0=\sigma_1=l$ , $\sigma_2 = 1-s$, $\beta_0=\beta_1=\half+$ , $ \beta_2 = \frac{3}{4}-$   , which  requires $3l-2s >-1$ and $4l-s > 0$ as before.
  \end{proof}                                             

\begin{proof}[Proof of (\ref{8.1})]
Using (\ref{2.7}) we obtain
\begin{align*}
\Pi_{\pm_0} (A_{k,df}^{\pm_2} \alpha^k \psi) & = \sum_{\pm} \Pi_{\pm_0} (A_{k,df}^{\pm_2} \alpha^k \Pi_{\pm} \psi)	\\
& =  \sum_{\pm} \Pi_{\pm_0} (A_{k,df}^{\pm_2} \Pi_{\mp}(\alpha^k \Pi_{\pm} \psi)) - \sum_{\pm} \Pi_{\pm_0} (A_{k,df}^{\pm_2} R^k_{\pm}\psi_{\pm}) = I + II \, .
\end{align*}
Concerning I we have to prove by duality
\begin{align*}
&	| \int \int \langle \Pi_{\pm_0} (A_{k,df}^{\pm_2} \Pi_{\mp_1}(\alpha^k \psi_{\pm_1})),\psi_{0_{\pm_0}} \rangle dx \, dt |\\
& \lesssim \|A_{k,df}^{\pm_2}\|_{X^{s,\frac{3}{4}+}_{\pm_2}} \|\psi_{1_{\pm_1}}\|_{X^{l,\half+}_{\pm_1}} \|\Pi_{\pm 0} \psi_{0_{\pm_0}}\|_{X^{-l,\half-}_{\pm_0}} \, .
\end{align*}                           
The left  hand side equals
\begin{align*}
		&| \int \int   (A_{k,df}^{\pm_2} \langle \Pi_{\mp_1}(\alpha^k \psi_{\pm_1})),\Pi_{\pm_0} \psi_{0_{\pm_0}} \rangle dx \, dt | \\
		&=  | \int \int   (A_{k,df}^{\pm_2} \langle \Pi_{\pm_0} \Pi_{\mp_1}(\alpha^k \psi_{\pm_1})),\psi_{0_{\pm_0}} \rangle dx \, dt | \,.                  
 \end{align*}                          
It contains a spinorial null form between    $\psi_{\pm_1}$ and $\psi_{0_{\pm_0}}$  as in (\ref{30a}), so that it remains to prove
\begin{equation}
	\label{71}
	\|B_{\pm_1,\pm_0}(\psi_{\pm_1},\psi_{0_{\pm_0}})\|_{X^{-s,-\frac{3}{4}-}_{\pm_2}} \lesssim \|\psi_{\pm_1}\|_{X^{l,\half+}_{\pm_1}} \|\psi_{0_{\pm_0}}\|_{X_{\pm_0}^{-l,\half--}} \, .
\end{equation}                     
We apply Prop. \ref{Prop.1.2} with parameters $\sigma_0 = s$ , $\sigma_1=l$ , $\sigma_2 = -l$ , $\beta_0=\frac{3}{4}+$,  $\beta_1= \half+$, $\beta_2= \half--$ .  We need $2s-l > 1$ and $s >\frac{3}{4}$ .

Concerning II we remark that
\begin{align*}
A^{df} &= |\nabla|^{-2} \nabla \times (\nabla \times A) =   |\nabla|^{-2} \nabla \times (\nabla \times A^{df}) + |\nabla|^{-2} \nabla \times (\nabla \times A^{cf}) \\ &=|\nabla|^{-2} \nabla \times (\nabla \times A^{df}) \, .                         
\end{align*}                           
This implies   
$$A_{l}^{df} R_{\pm_1} \psi_{\pm_1} = \epsilon^{lkm} \partial_k w_m R^l_{\pm_1} \psi_{\pm_1}  = (\nabla w_m \times \frac{\nabla}{|\nabla|} \psi_{\pm_1})^m \, , $$  
where $\epsilon^{lkm}$ denotes the Levi-Civita symbol with $\epsilon^{123} = 1$ and $w= |\nabla|^{-2} \nabla \times A^{df}$  , so that $\partial_j w_m = |\nabla|^{-2} \partial_j \partial_k A^{df}_l \epsilon^{lkm}$ .    This is a $Q_{ij}$-type null form between $w_m$ and $ |\nabla|^{-1} \psi_{\pm_1}$ , so that we have to prove
$$\|B_{\pm_1,\pm_2}(A^{df,\pm_2}_l ,\psi_{\pm_1})\|_{X^{l,-\half++}_{\pm}} \lesssim    \|A^{df,\pm_2}_l \|_{X^{s,\frac{3}{4}+}_{\pm_2}} \|\psi_{\pm_1}\|_{X^{l,\half+}_{\pm_1}} \, . $$  
This is implied by   Prop. \ref{Prop.1.2} with parameters $\sigma_0 = -l$ , $\sigma_1=l$ , $\sigma_2 = s$ , $\beta_0=\half --$, $\beta_1= \half+$ , $\beta_2= \frac{3}{4}+$ , if $2s-l > 1$ and $s >\frac{3}{4}$ .
\end{proof}    
                                    
The estimates (\ref{16})-(\ref{18}) have been essentially given by Tao \cite{T1}. For the sake of completeness we give the details. We remark that it is especially (\ref{18}) which prevents a large data result, because it seems to be difficult to replace $X^{s+\frac{1}{4},-\frac{1}{2}+\epsilon}_{\tau=0}$ by $X^{s+\frac{1}{4},-\frac{1}{2}+\epsilon+}_{\tau=0}$ on the left hand side.\\
\begin{proof}[Proof of (\ref{17})]
As usual the singularity of $|\nabla|^{-1}$ is harmless in dimension 3 (\cite{T}, Cor. 8.2) and it can be replaced by $\langle \nabla \rangle^{-1}$. Taking care of the time derivative we reduce to
\begin{align*}
	\big|\int \int u_1 u_2 u_3 dx dt\big| \lesssim \|u_1\|_{X^{s+\frac{1}{4},\frac{1}{2}+\epsilon}_{\tau =0}}
	\|u_2\|_{X^{s+\frac{1}{4},-\frac{1}{2}+\epsilon}_{\tau =0}}
	\|u_3\|_{X^{\frac{3}{4} -s,\frac{1}{2}-2\epsilon+}_{\tau =0}} \, ,
\end{align*}
which follows from Sobolev's multiplication rule, because  $s>\frac{1}{4}$ . 
\end{proof}
\begin{proof}[Proof of (\ref{18})]
a. If $\widehat{\phi}$ is supported in  $ ||\tau|-|\xi|| \gtrsim |\xi| $ , we obtain
$$ \|\phi\|_{X^{s+\frac{1}{4},\frac{1}{2}+\epsilon}_{\tau=0}} \lesssim \|\phi\|_{X^{s,\frac{3}{4}+\epsilon}_{|\tau|=|\xi|}} \,. $$
Thus (\ref{18}) follows from (\ref{17}).\\
b. It remains to show
$$ \big|\int\int (uv_t w + uvw_t) dxdt \big| \lesssim 
\|u\|_{X^{\frac{3}{4}-s,\frac{1}{2}-\epsilon}_{\tau =0}}
\|w\|_{X^{s,\frac{3}{4}+\epsilon}_{|\tau| =|\xi|}}
\|v\|_{X^{s+\frac{1}{4}-\epsilon,\frac{1}{2}+\epsilon}_{\tau =0}} \, $$
whenever $\widehat w$ is supported in $||\tau|-|\xi|| \ll |\xi|$.
This is equivalent to
$$ \int_* m(\xi_1,\xi_2,\xi_3,\tau_1,\tau_2,\tau_3) \prod_{i=1}^3 \widehat{u}_i(\xi_i,\tau_i) d\xi d\tau \lesssim \prod_{i=1}^3 \|u_i\|_{L^2_{xt}} \, $$
where $d\xi = d\xi_1 d\xi_2 d\xi_3$ , $d\tau = d\tau_1 d\tau_2 d\tau_n$ and * denotes integration over $\sum_{i=1}^3 \xi_i = \sum_{i=1}^3 \tau_i = 0$. The Fourier transforms are nonnegative without loss of generality. Here
$$ m= \frac{(|\tau_2|+|\tau_3|) \chi_{||\tau_3|-|\xi_3|| \ll |\xi_3|}}{\langle \xi_1 \rangle^{\frac{3}{4}-s} \langle \tau_1 \rangle^{\frac{1}{2}-\epsilon} \langle \xi_2 \rangle^{s+\frac{1}{4}-\epsilon} \langle \tau_2 \rangle^{\frac{1}{2}+\epsilon} \langle \xi_3 \rangle^s \langle |\tau_3|-|\xi_3|\rangle^{\frac{3}{4}+\epsilon}} \, . $$
Since $\langle \tau_3 \rangle \sim \langle \xi_3 \rangle$ and $\tau_1+\tau_2+\tau_3=0$ we have 
\begin{equation}
	\label{N4'}
	|\tau_2| + |\tau_3| \lesssim \langle \tau_1 \rangle^{\frac{1}{2}-\epsilon} \langle \tau_2 \rangle^{\frac{1}{2}+\epsilon} +\langle \tau_1 \rangle^{\frac{1}{2}-\epsilon} \langle \xi_3 \rangle^{\frac{1}{2}+\epsilon} +\langle \tau_2 \rangle^{\frac{1}{2}+\epsilon} \langle \xi_3 \rangle^{\frac{1}{2}-\epsilon} , 
\end{equation}
so that concerning the first term on the right hand side of (\ref{N4'}) we have to show
$$\big|\int\int uvw dx dt\big| \lesssim \|u\|_{X^{\frac{3}{4}-s,0}_{\tau=0}} \|v\|_{X^{s+\frac{1}{4}-\epsilon,0}_{\tau=0}} \|w\|_{X^{s,\frac{3}{4}+\epsilon}_{|\tau|=|\xi|}} \ , $$
which follows from Sobolev's multiplication rule, because $s> \half$. This is sharp with respect to the time derivative. As a consequence we need the smallness assumption on the data for local existence.\\
Concerning the second term on the right hand side of (\ref{N4'}) we use $\langle \xi_1 \rangle^{s-\frac{3}{4}} \lesssim \langle \xi_2 \rangle^{s-\frac{3}{4}} + \langle \xi_3 \rangle^{s-\frac{3}{4}}$, so that we reduce to
\begin{equation}
	\label{51}
	\big|\int\int uvw dx dt\big|  
	\lesssim\|u\|_{X^{0,0}_{\tau=0}} \|v\|_{X^{1-\epsilon,\frac{1}{2}+\epsilon}_{\tau=0}} \|w\|_{X^{s-\frac{1}{2}-\epsilon,\frac{3}{4}+\epsilon}_{|\tau|=|\xi|}} 
\end{equation}
and 
\begin{equation}
	\label{52}
	\big|\int\int uvw dx dt\big|  
	\lesssim\|u\|_{X^{0,0}_{\tau=0}}
	\|v\|_{X^{s+\frac{1}{4}-\epsilon,\frac{1}{2}+\epsilon}_{\tau=0}} \|w\|_{X^{\frac{1}{4}-\epsilon,\frac{3}{4}+\epsilon}_{|\tau|=|\xi|}} \,.
\end{equation}
(\ref{51}) is implied by Sobolev and (\ref{Tao}) as follows:
$$ \big| \int\int uvw dx dt \big| \le \|u\|_{L^2_xL^2_t} \|v\|_{L^4_x L^{\infty}_t} \|w\|_{L^4_xL^2_t}
\lesssim \|u\|_{X^{0,0}_{\tau =0}} \|v\|_{X^{1-\epsilon,\frac{1}{2}+}_{\tau = 0}} \|w\|_{X^{\frac{1}{4},\frac{1}{2}+}_{|\tau|=|\xi|}} $$
For (\ref{52}) we obtain 
\begin{align*}
 \big| \int\int uvw dx dt \big| & \le \|u\|_{L^2_x L^2_t} \|v\|_{L^q_x L^{\infty}_t} \|w\|_{L^p_x L^2_t}	\,
\end{align*}
where $\frac{1}{q}=\frac{1}{4}-\epsilon$ and $\frac{1}{p} =\frac{1}{4}+\epsilon$ .  For $s >\frac{3}{4}$ we obtain by Sobolev  $$\|v\|_{L^q_x L^{\infty}_t} \lesssim\|v\|_{X^{s+\frac{1}{4}-\epsilon,\frac{1}{2}+\epsilon}_{\tau=0}} \, .$$ 
Interpolation between (\ref{Tao}) $\|w\|_{L^4_x L^2_t} \lesssim \|w\|_{X_{|\tau|=|\xi|}^{\frac{1}{4},\half+}}$ and the trivial identity $\|w\|_{L^2_x L^2_t} = \|w\|_{X^{0,0}_{|\tau|=|\xi|}}$ implies
$$\|w\|_{L^p_x L^2_t} \lesssim \|w\|_{X^{\frac{1}{4}-\epsilon,\half+}_{|\tau|=|\xi|}} \, $$
so that (\ref{52}) follows. 

Concerning the last term on the right hand side of (\ref{N4'}) we use
$\langle \xi_1 \rangle^{s-\frac{3}{4}} \lesssim \langle \xi_2 \rangle^{s-\frac{3}{4}} + \langle \xi_3 \rangle^{s-\frac{3}{4}}$ so that we reduce to
\begin{equation}
	\label{53}
	\big|\int\int uvw dx dt\big| 
	\lesssim \|u\|_{X^{0,\frac{1}{2}-\epsilon}_{\tau=0}} 
	\|v\|_{X^{1-\epsilon,0}_{\tau=0}}
	\|w\|_{X^{s-\frac{1}{2}+\epsilon,\frac{3}{4}+\epsilon}_{|\tau|=|\xi|}} 
\end{equation}
and
\begin{equation}
	\label{54}
	\big|\int\int uvw dx dt\big| 
	\lesssim \|u\|_{X^{0,\frac{1}{2}-\epsilon}_{\tau=0}} 
	\|v\|_{X^{s+\frac{1}{4}-\epsilon,0}_{\tau=0}} \|w\|_{X^{\frac{1}{4}+\epsilon,\frac{3}{4}+\epsilon}_{|\tau|=|\xi|}}  \,.
\end{equation}
We estimate as follows:
\begin{align}
	\nonumber
	\big| \int\int uvw dx dt \big| &\lesssim \|u\|_{L^2_x L^{\frac{1}{\epsilon}-}_t} \|v\|_{L^{p}_x L^2_t} \|w\|_{L^{r}_x L^{q}_t} \\
	\label{54'}
	&\lesssim \|u\|_{X^{0,\frac{1}{2}-\epsilon}_{\tau=0}} 
	\|v\|_{X^{1-\epsilon,0}_{\tau=0}}
	\|w\|_{X^{\frac{1}{6}+\epsilon+,\frac{1}{2}+}_{|\tau|=|\xi|}} \, ,
\end{align}
which would be sufficient for (\ref{53}) and (\ref{54}) under our assumption $ s > \frac{3}{4} $ . For the proof of (\ref{54'}) we choose $\frac{1}{q}=\half-\epsilon+$ , $\frac{1}{p}=\frac{1}{6}+\frac{\epsilon}{3}$ and $\frac{1}{r}=\frac{1}{3}-\frac{\epsilon}{3}$  so that  $\|u\|_{L^2_x L^{\frac{1}{\epsilon}-}_t} \lesssim \|u\|_{X^{0,\half-\epsilon}}$ and  by Sobolev  $\|v\|_{L^{p}_x L^2_t} \lesssim \|v\|_{X^{1-\epsilon,0}_{\tau=0}}$ . Moreover  interpolation between (\ref{Tao}) $\|w\|_{L^4_x L^2_t} \lesssim \|w\|_{X^{\frac{1}{4},\half+}_{|\tau|=|\xi|}}$ and the trivial identity $\|w\|_{L^2_x L^2_t} = \|w\|_{X^{0,0}_{|\tau|=|\xi|}}$ implies 
$$\|w\|_{L^r_x L^2_t} \lesssim \|w\|_{X^{\half-\frac{1}{r},\half+}_{|\tau|=|\xi|}} \, , $$ 
Interpolation between Strichartz' inequality (\ref{15}) $\|w\|_{L^4_{xt}} \lesssim \|w\|_{X^{\half,\half+}_{|\tau|=|\xi|}}$ and the trivial identity gives 
$$\|w\|_{L^r_x L^r_t} \lesssim \|w\|_{X^{1-\frac{2}{r},\half+}_{|\tau|=|\xi|}} \,, $$
so that another interpolation between these estimates implies the following bound for the last factor
$$\|w\|_{L^r_x L^q_t} \lesssim \|w\|_{X^{\frac{1}{6}+\epsilon+,\half+}_{|\tau|=|\xi|}} \, , $$
as one easily checks. This implies (\ref{54'}).
\end{proof}
 
\begin{proof}[Proof of (\ref{16})]
If $\widehat{\phi}$ is supported in $||\tau|-|\xi|| \gtrsim |\xi|$ we obtain
$$\|\phi\|_{X^{s+\frac{1}{4},\frac{1}{2}+\epsilon}_{\tau =0}}
\lesssim \|\phi\|_{X^{s,\frac{3}{4}+\epsilon}_{|\tau|=|\xi|}} \, . $$
which implies that (\ref{16}) follows from (\ref{18}), if $\widehat{\phi}_1$ or $\widehat{\phi}_2$ have this support property. So we may assume that both functions are supported in $||\tau|-|\xi|| \ll |\xi|$. This means that it suffices to show
$$ \int_* m(\xi_1,\xi_2,\xi_3,\tau_1,\tau_2,\tau_3) \prod_{i=1}^3 \widehat{u}_i(\xi_i,\tau_i) d\xi d\tau \lesssim \prod_{i=1}^3 \|u_i\|_{L^2_{xt}} \, , $$
where
$$m= \frac{|\tau_3|\chi_{||\tau_2|-|\xi_2|| \ll |\xi_2|} \chi_{||\tau_3|-|\xi_3|| \ll |\xi_3|}}{\langle \xi_1 \rangle^{\frac{3}{4}-s} \langle \tau_1 \rangle^{\frac{1}{2}-\epsilon-} \langle \xi_2 \rangle^s \langle |\tau_2|-|\xi_2| \rangle^{\frac{3}{4}+\epsilon} \langle \xi_3 \rangle^s \langle |\tau_3|-|\xi_3|\rangle^{\frac{3}{4}+\epsilon}} \, . $$
Since $\langle \tau_3 \rangle \sim \langle \xi_3 \rangle$ , $\langle \tau_2 \rangle \sim \langle \xi_2 \rangle$ and $\tau_1+\tau_2+\tau_3=0$ we have 
\begin{equation}
	|\tau_3| \lesssim \langle \tau_1 \rangle^{\frac{1}{2}-\epsilon-} \langle \xi_3 \rangle^{\frac{1}{2}+\epsilon+} +\langle \xi_2 \rangle^{\frac{1}{2}-\epsilon-} \langle \xi_3 \rangle^{\frac{1}{2}+\epsilon+} , 
\end{equation}
Concerning the first term on the right hand side we have to show 
$$\big|\int \int uvw dx dt\big| \lesssim \|u\|_{X^{\frac{3}{4}-s,0}_{\tau=0}} \|v\|_{X^{s,\frac{3}{4}+\epsilon}_{|\tau|=|\xi|}} \|w\|_{X^{s-\frac{1}{2}-\epsilon-,\frac{3}{4}+\epsilon}_{|\tau|=|\xi|}} \, .$$
We use Prop. \ref{Prop.1.2} , which shows
$$\|vw\|_{L^2_t H^{s-\frac{3}{4}}_x} \lesssim \|v\|_{X^{s,\frac{1}{2}+}_{|\tau|=|\xi|}} \|w\|_{X^{s-\frac{1}{2}-\epsilon-,\frac{1}{2}+}_{|\tau|=|\xi|}} $$
under the assumption $s > \frac{3}{4}$. \\
Concerning the second term on the right hand side we use $\langle \xi_1 \rangle^{s-\frac{3}{4}} \lesssim \langle \xi_2 \rangle^{s-\frac{3}{4}} + \langle \xi_3 \rangle^{s-\frac{3}{4}}$ , so that we reduce to
$$
\big|\int\int uvw dx dt\big| 
\lesssim \|u\|_{X^{0,\frac{1}{2}-\epsilon-}_{\tau=0}} 
\|v\|_{X^{\frac{1}{4}+\epsilon+,\frac{3}{4}+\epsilon}_{|\tau|=|\xi|}}
\|w\|_{X^{s-\frac{1}{2}-\epsilon-,\frac{3}{4}+\epsilon}_{|\tau|=|\xi|}} 
$$
and
$$
\big|\int\int uvw dx dt\big| 
\lesssim \|u\|_{X^{0,\frac{1}{2}-\epsilon-}_{\tau=0}}
\|v\|_{X^{s-\frac{1}{2}+\epsilon+,\frac{3}{4}+\epsilon}_{|\tau|=|\xi|}}  \|w\|_{X^{\frac{1}{4}-\epsilon-,\frac{3}{4}+\epsilon}_{|\tau|=|\xi|}} \, .
$$
We obtain
\begin{align*}
\big|\int\int uvw dx dt\big| & \lesssim \|u\|_{L^2_x L^{\frac{1}{2\epsilon}-}_t} \|v\|_{L^{r+}_x L^{q+}_t} \|w\|_{L^{p-}_x L^2_t} \\
&
\lesssim \|u\|_{X^{0,\half-2\epsilon+}_{\tau=0}}
\|v\|_{X^{\frac{1}{4}+5\epsilon+,\half+}_{|\tau|=|\xi|}}  \|w\|_{X^{\frac{1}{4}-\epsilon-,\half+}_{|\tau|=|\xi|}} \, ,
\end{align*}
which implies both estimates for $s>\frac{3}{4}$ and $\epsilon>0$ sufficiently small.. Here we choose $\frac{1}{r}= \frac{1}{4}-\epsilon$ , $\frac{1}{q}= \half-2\epsilon$ and $\frac{1}{p}=\frac{1}{4}+\epsilon$ .  Here for the first factor we interpolated between $\|u\|_{L^2_x L^{\infty}_2} \lesssim \|u\|_{X^{0,\half+}_{|\tau|=|\xi|}}$ and the trivial identity $\|u\|_{L^2_{xt}} = \|u\|_{X^{0,0}_{|\tau|=|\xi|}}$, for the second factor between (\ref{Tao}) $\|v\|_{L^4_x L^2_t} \lesssim \|v\|_{X^{\frac{1}{4},\half+}_{|\tau|=|\xi|}}$ and the Sobolev type inequality $\|v\|_{L^{\infty}_{xt}} \lesssim \|v\|_{X^{\frac{3}{2}+,\half+}_{|\tau|=|\xi|}}$ , and for the last factor between (\ref{Tao}) and the trivial identity, both  with interpolation parameter $\theta = 1-4\epsilon-$ .
\end{proof}

\begin{proof}[Proof of (\ref{17a})]
By the fractional Leibniz rule we have to prove
$$\|uv\|_{X^{0,-\half+\epsilon+}_{\tau=0}} \lesssim \|u\|_{X^{l-s+\frac{3}{4},\half+}_{|\tau|=|\xi|}}
 \|v\|_{X^{l,\half+}_{|\tau|=|\xi|}} \, , $$
 which is equivalent to
 $$| \int \int uvw \, dx\, dt| \lesssim \|u\|_{X^{l-s+\frac{3}{4},\half+}_{|\tau|=|\xi|}}
 \|v\|_{X^{l,\half+}_{|\tau|=|\xi|}} \|w\|_{X^{0,\half-\epsilon-}_{\tau=0}} \, . $$
By H\"older's inequality we obtain:
 $$| \int \int uvw \, dx\, dt| \lesssim \|u\|_{L^4_x L^2_t}
\|v\|_{L^4_x L^{\frac{2}{1-2\epsilon}+}_t} \|w\|_{L^2_x L^{\frac{1}{\epsilon}-}_t} \, . $$
The first factor is estimated by (\ref{Tao}) using the assumption $l-s\ge -\half$ and the last factor by Sobolev. For the second factor we interpolate between (\ref{Tao}) $\|v\|_{L^4_x L^2_t} \lesssim \|v\|_{X^{\frac{1}{4},\half+}_{|\tau|=|\xi|}}$ and Strichartz' estimate $\|v\|_{L^4_x L^4_t} \lesssim \|v\|_{X^{\half,\half+}_{|\tau|=|\xi|}}$ , which implies $\|v\|_{L^4_x L^{\frac{2}{1-2\epsilon}+}_t} \lesssim \|v\|_{X^{\frac{1}{4}+\epsilon+,\half+}_{|\tau|=|\xi|}} \lesssim \|v\|_{X^{l,\half+}_{|\tau|=|\xi|}}$ for $l > \frac{1}{4}$ ,
 so that (\ref{17a}) is proven. 
\end{proof}
\begin{proof}[Proof of (\ref{19})] Sobolev's multiplication law shows the estimate
$$ \| |\nabla|^{-1} (A_1 \partial_t A_2)\|_{C^0(H^{s-1})} \lesssim \|A_1\|_{C^0(H^s)} \|\partial_t A_2\|_{C^0(H^{s-1})}$$
for $s > \half$. Use now $$ A=A^{cf} + \sum_{\pm} A^{df}_{\pm} \quad , \quad \partial_t A = \partial_t A^{cf} + i \langle \nabla \rangle(A_+^{df} -A_-^{df}) \, $$
from which the estimate (\ref{19}) easily follows. 
\end{proof}
\begin{proof}[Proof of (\ref{19a})] 
This reduces to the estimate
$$\| \psi_1 \psi_2\|_{C^0(H^{s-2})} \lesssim \|\psi_1\|_{C^0(H^l)} \|\psi_s\|_{C^0(H^l)} \, , $$
which by the Sobolev multiplication law requires $2-s+2l > \frac{3}{2}$ . This is implied by our assumption $l-s \ge - \half$ and $l>\frac{1}{4}$ . 
\end{proof}

\begin{proof}[Proof of (\ref{29})]
	This a variant of a proof given by  Tao (\cite{T1}) for the Yang-Mills case.
	  We have to show
$$
\int_* m(\xi,\tau) \prod_{i=1}^3 \widehat{u}_i(\xi_i,\tau_i)  d\xi d\tau \lesssim \prod_{i=1}^3 \|u_i\|_{L^2_{xt}} \, , 
$$
where $\xi=(\xi_1,\xi_2,\xi_3) \, , \,\tau=(\tau_1,\tau_2,\tau_3)$ , * denotes integration over $ \sum_{i=1}^3 \xi_i = \sum_{i=1}^3 \tau_i = 0$ , and 
$$ m = \frac{(|\xi_2|+|\xi_3|) \langle \xi_1 \rangle^{s-1} \langle |\tau_1|-|\xi_1|) \rangle^{-\frac{1}{4}+2\epsilon}}{\langle \xi_2 \rangle^s \langle |\tau_2| - |\xi_2|\rangle^{\frac{3}{4}+\epsilon}  \langle \xi_3 \rangle^{s+\frac{1}{4}}\langle \tau_3 \rangle^{\frac{1}{2}+\epsilon}} \, .$$
Case 1: $|\xi_2| \le |\xi_1|$ ($\Rightarrow$ $|\xi_2|+|\xi_3| \lesssim |\xi_1|$). \\
By two applications of the averaging principle (\cite{T}, Prop. 5.1) we may replace $m$ by
$$ m' = \frac{ \langle \xi_1 \rangle^s \chi_{||\tau_2|-|\xi_2||\sim 1} \chi_{|\tau_3| \sim 1}}{ \langle \xi_2 \rangle^s \langle \xi_3 \rangle^{s+\frac{1}{4}}} \, . $$
Let now $\tau_2$ be restricted to the region $\tau_2 =T + O(1)$ for some integer $T$. Then $\tau_1$ is restricted to $\tau_1 = -T + O(1)$, because $\tau_1 + \tau_2 + \tau_3 =0$, and $\xi_2$ is restricted to $|\xi_2| = |T| + O(1)$. The $\tau_1$-regions are essentially disjoint for $T \in {\mathbb Z}$ and similarly the $\tau_2$-regions. Thus by Schur's test (\cite{T}, Lemma 3.11) we only have to show
\begin{align*}
	&\sup_{T \in {\mathbb Z}} \int_* \frac{\langle \xi_1 \rangle^s \chi_{\tau_1=-T+O(1)} \chi_{\tau_2=T+O(1)} \chi_{|\tau_3|\sim 1} \chi_{|\xi_2|=|T|+O(1)}}{\langle \xi_2 \rangle^s \langle \xi_3 \rangle^{s+\frac{1}{4}}} \prod_{i=1} \widehat{u}_i(\xi_i,\tau_i)  d\xi d\tau  \\
	& \hspace{25em} \lesssim \prod_{i=1}^3 \|u_i\|_{L^2_{xt}} \, . 
\end{align*}
The $\tau$-behaviour of the integral is now trivial, thus we reduce to
\begin{equation}
	\label{55}
	\sup_{T \in {\mathbb N}} \int_{\sum_{i=1}^3 \xi_i =0}  \frac{ \langle \xi_1 \rangle^s \chi_{|\xi_2|=|T|+O(1)}}{ \langle T \rangle^s \langle \xi_3 \rangle^{s+\frac{1}{4}}} \widehat{f}_1(\xi_1)\widehat{f}_2(\xi_2)\widehat{f}_3(\xi_3)d\xi \lesssim \prod_{i=1}^3 \|f_i\|_{L^2_x} \, .
\end{equation}
Assuming now $|\xi_3| \le |\xi_1|$ (the other case being simpler) 
it only remains to consider the following two cases: \\
Case 1.1: $|\xi_1| \sim |\xi_3| \gtrsim T$. We obtain in this case
\begin{align*}
	L.H.S. \, of \, (\ref{55}) 
	&\lesssim \sup_{T \in{\mathbb N}} \frac{1}{T^{s+\frac{1}{4}}} \|f_1\|_{L^2} \|f_3\|_{L^2} \| {\mathcal F}^{-1}(\chi_{|\xi|=T+O(1)} \widehat{f}_2)\|_{L^{\infty}({\mathbb R}^3)} \\
	&\lesssim \sup_{T \in{\mathbb N}} \frac{1}{ T^{s+\frac{1}{4}}} 
	\|f_1\|_{L^2} \|f_3\|_{L^2} \| \chi_{|\xi|=T+O(1)} \widehat{f}_2\|_{L^1({\mathbb R}^3)} \\
	&\lesssim \hspace{-0.1em}\sup_{T \in {\mathbb N}} \frac{T}{T^{s+\frac{1}{4}}}  \prod_{i=1}^3 \|f_i\|_{L^2} \lesssim\hspace{-0.1em}
	\prod_{i=1}^3 \|f_i\|_{L^2} \, ,
\end{align*}
provided $s \ge \frac{3}{4}$ .\\
Case 1.2: $|\xi_1| \sim T \gtrsim |\xi_3|$. 
An elementary calculation shows that
\begin{align*}
	L.H.S. \, of \,  (\ref{55})
	\lesssim \sup_{T \in{\mathbb N}} \| \chi_{|\xi|=T+O(1)} \ast \langle \xi \rangle^{-2(s+\frac{1}{4})}\|^{\frac{1}{2}}_{L^{\infty}(\mathbb{R}^2)} \prod_{i=1}^3 \|f_i\|_{L^2_x} \lesssim \prod_{i=1}^3 \|f_i\|_{L^2_x} 
\end{align*}
for $s > \frac{3}{4}$ ,
so that the desired estimate follows.\\
Case 2. $|\xi_1| \le |\xi_2|$ ($\Rightarrow$ $|\xi_2|+|\xi_3| \lesssim |\xi_2|$). \\
Exactly as in case 1 we reduce to
$$
\sup_{T \in {\mathbb N}} \int_{\sum_{i=1}^3 \xi_i =0}  \frac{ \langle \xi_1 \rangle^{s-1} \chi_{|\xi_2|=|T|+O(1)}}{ \langle T \rangle^{s-1} \langle \xi_3 \rangle^{s+\frac{1}{4}}} \widehat{f}_1(\xi_1)\widehat{f}_2(\xi_2)\widehat{f}_2(\xi_3)d\xi \lesssim \prod_{i=1}^3 \|f_i\|_{L^2_x} \, .
$$
This can be treated as in case 1.
\end{proof}

\begin{proof}[Proof of (\ref{30})]
	By the Sobolev multiplication law we obtain
	$$ |\int \int fgh \, dx\, dt| \lesssim \|f\|_{X^{s+\frac{1}{4},\half+\epsilon}_{\tau=0}} \|g\|_{X^{s-\frac{3}{4},\half+\epsilon}_{\tau=0}}  \|h\|_{X^{\frac{5}{4}-s-2\epsilon,-\half}_{\tau=0}} $$
	for $s > \frac{3}{4}$ . The elementary estimate $\langle \xi \rangle^{\frac{1}{4}-2\epsilon} \langle \tau \rangle^{-(\frac{1}{4}-2\epsilon)} \lesssim  \langle |\tau|-|\xi| \rangle^{\frac{1}{4}-2\epsilon}$ implies
	$$  \|h\|_{X^{\frac{5}{4}-s-2\epsilon,-\half}_{\tau=0}} \lesssim \|h\|_{X^{1-s,\frac{1}{4}-2\epsilon}_{|\tau|=|\xi|}} \,, $$
	thus the claimed estimate.
\end{proof}

\begin{proof}[Proof of (\ref{31})]
	We may assume $\frac{3}{4}<s\le 1$ , because the case $s>1$ follows easily by the fractional Leibniz rule.
	We obtain $$\|A\|_{L^{4-}_t L^2_x} \lesssim \|A\|_{X^{0,\frac{1}{4}-}_{|\tau|=|\xi|}} \lesssim \|A\|_{X^{1-s,\frac{1}{4}-}_{|\tau|=|\xi|}}\, $$
	so that by duality
	$$	\|A_1 A_2 A_3\|_{X^{s-1,-\frac{1}{4}+}_{|\tau|=|\xi|}} \lesssim \|A_1 A_2 A_3\|_{L^{\frac{4}{3}+}_t L^2_x} \lesssim \prod_{i=1}^3 \|A_i\|_{L^{4+}_t L^6_x} \, . $$
	Now by Sobolev 
	$$ \|A_i\|_{L^{4+}_t L^6_x} \lesssim \|A_i\|_{L^{4+}_t H^1_x} \lesssim \|A_i\|_{X^{s+\frac{1}{4},\half+}_{\tau=0}} $$
	and by Strichartz  and Sobolev as well
	$$\|A_i\|_{L^{4+}_t L^6_x} \lesssim \|A_i\|_{L^{4+}_t H^{\frac{1}{4},4+}_x} \lesssim \|A_i\|_{X^{\frac{3}{4}+,\half+}_{|\tau|=|\xi|}} \lesssim \|A_i\|_{X^{s,\half+}_{|\tau|=|\xi|}} \, ,$$
	which implies the claim.

\end{proof}

\begin{proof}[Proof of (\ref{43})]
	We apply again  Tao's method (\cite{T1}). We have to show
$$
\int_* m(\xi,\tau) \prod_{i=1}^3 \widehat{u}_i(\xi_i,\tau_i)  d\xi d\tau \lesssim \prod_{i=1}^3 \|u_i\|_{L^2_{xt}} \, , 
$$
where 
$$ m = \frac{ \langle \xi_1 \rangle^{l} \langle |\tau_1|-|\xi_1|) \rangle^{-\frac{1}{2}++}}{\langle \xi_2 \rangle^l \langle |\tau_2| - |\xi_2|\rangle^{\half+}  \langle \xi_3 \rangle^{s+\frac{1}{4}}\langle \tau_3 \rangle^{\frac{1}{2}+}} \, .$$
By two applications of the averaging principle (\cite{T}, Prop. 5.1) we may replace $m$ by
$$ m' = \frac{ \langle \xi_1 \rangle^l \chi_{||\tau_2|-|\xi_2||\sim 1} \chi_{|\tau_3| \sim 1}}{ \langle \xi_2 \rangle^l \langle \xi_3 \rangle^{s+\frac{1}{4}}} \, . $$
Let now $\tau_2$ be restricted to the region $\tau_2 =T + O(1)$ for some integer $T$. Then $\tau_1$ is restricted to $\tau_1 = -T + O(1)$, because $\tau_1 + \tau_2 + \tau_3 =0$, and $\xi_2$ is restricted to $|\xi_2| = |T| + O(1)$. The $\tau_1$-regions are essentially disjoint for $T \in {\mathbb Z}$ and similarly the $\tau_2$-regions. Thus by Schur's test (\cite{T}, Lemma 3.11) we only have to show
\begin{align*}
	&\sup_{T \in {\mathbb Z}} \int_* \frac{\langle \xi_1 \rangle^l \chi_{\tau_1=-T+O(1)} \chi_{\tau_2=T+O(1)} \chi_{|\tau_3|\sim 1} \chi_{|\xi_2|=|T|+O(1)}}{\langle \xi_2 \rangle^l \langle \xi_3 \rangle^{s+\frac{1}{4}}} \prod_{i=1} \widehat{u}_i(\xi_i,\tau_i)  d\xi d\tau  \\
	& \hspace{25em} \lesssim \prod_{i=1}^3 \|u_i\|_{L^2_{xt}} \, . 
\end{align*}
The $\tau$-behaviour of the integral is now trivial, thus we reduce to
\begin{equation}
	\label{55'}
	\sup_{T \in {\mathbb N}} \int_{\sum_{i=1}^3 \xi_i =0}  \frac{ \langle \xi_1 \rangle^l \chi_{|\xi_2|=|T|+O(1)}}{ \langle T \rangle^l \langle \xi_3 \rangle^{s+\frac{1}{4}}} \widehat{f}_1(\xi_1)\widehat{f}_2(\xi_2)\widehat{f}_3(\xi_3)d\xi \lesssim \prod_{i=1}^3 \|f_i\|_{L^2_x} \, .
\end{equation}
Assuming now $|\xi_3| \le |\xi_1|$ (the other case being simpler) 
it only remains to consider the following two cases: \\
Case 1.1: $|\xi_1| \sim |\xi_3| \gtrsim T$. We obtain in this case
\begin{align*}
	L.H.S. \, of \, (\ref{55'}) 
	&\lesssim \sup_{T \in{\mathbb N}} \frac{1}{T^{s+\frac{1}{4}}} \|f_1\|_{L^2} \|f_3\|_{L^2} \| {\mathcal F}^{-1}(\chi_{|\xi|=T+O(1)} \widehat{f}_2)\|_{L^{\infty}({\mathbb R}^3)} \\
	&\lesssim \sup_{T \in{\mathbb N}} \frac{1}{ T^{s+\frac{1}{4}}} 
	\|f_1\|_{L^2} \|f_3\|_{L^2} \| \chi_{|\xi|=T+O(1)} \widehat{f}_2\|_{L^1({\mathbb R}^3)} \\
	&\lesssim \hspace{-0.1em}\sup_{T \in {\mathbb N}} \frac{T}{T^{s+\frac{1}{4}}}  \prod_{i=1}^3 \|f_i\|_{L^2} \lesssim\hspace{-0.1em}
	\prod_{i=1}^3 \|f_i\|_{L^2} \, ,
\end{align*}
because $s \ge \frac{3}{4}$ .\\
Case 1.2: $|\xi_1| \sim T \gtrsim |\xi_3|$. 
An elementary calculation shows that
\begin{align*}
	L.H.S. \, of \,  (\ref{55'})
	\lesssim \sup_{T \in{\mathbb N}} \| \chi_{|\xi|=T+O(1)} \ast \langle \xi \rangle^{-2(s+\frac{1}{4})}\|^{\frac{1}{2}}_{L^{\infty}(\mathbb{R}^2)} \prod_{i=1}^3 \|f_i\|_{L^2_x} \lesssim \prod_{i=1}^3 \|f_i\|_{L^2_x} \, ,
\end{align*}
using that $2(s+\frac{1}{4}) > 2$ ,
so that the desired estimate follows.
\end{proof}

\section{Removal of the assumption $A^{cf}(0)=0$}
Applying an idea of Keel and Tao \cite{T1} we use the gauge invariance of the Yang-Mills-Dirac system to show that the condition $A^{cf}(0)=0$, which had to be assumed in Prop. \ref{Prop}, can be removed. 
\begin{lemma}
	\label{Lemma}
	 Let $s>\frac{3}{4}$ and $0 < \epsilon \ll 1$. Assume 
	$(A,\psi)\in( C^0([0,1],H^s) \cap C^1([0,1],H^{s-1}) \times (C^0([0,1],H^{l})$ , $A_0 = 0$ and
	\begin{equation}
		\label{***}
		\|A^{df}(0)\|_{H^s} + \|(\partial_t A)^{df}(0)\|_{H^{s-1}} + \|A^{cf}(0)\|_{H^s} + \|\psi(0)\|_{H^l}  \le \epsilon \, .
	\end{equation}
	Then there exists a gauge transformation $T$ preserving the temporal gauge such that $(TA)^{cf}(0) = 0$ and
	\begin{align}
		\label{T1}
		\|(TA)^{df}(0)\|_{H^s} + \|(\partial_t TA)^{df}(0)\|_{H^{s-1}} + \|(T \psi)(0)\|_{H^l} \lesssim \epsilon \, .
	\end{align}
	T preserves also the regularity, i.e. $TA\in C^0([0,1],H^s) \cap C^1([0,1],H^{s-1})$ , $T\psi \in C^0([0,1],H^l)$. If
	$A \in X^{s,\frac{3}{4}+}_+[0,1] + X^{s,\frac{3}{4}+}_-[0,1] + X^{s+\frac{1}{4},\frac{1}{2}+}_{\tau=0}[0,1]$,  $\partial_t A^{cf} \in C^0([0,1],\\H^{s-1})$ and $\psi\in X^{l,\half+}_+[0,1] + X^{l,\half+}_-[0,1]$ , then $TA$ , $T\psi$ belong to the same spaces. Its inverse $T^{-1}$ has the same properties.
\end{lemma}

\begin{proof}[Proof of Lemma \ref{Lemma}]
	For details of the proof we refer to a similar result for the Yang-Mills and Yang-Mills-Higgs equation in in \cite{P}, Lemma 4.1 (cf. also the sketch of proof in \cite{T1}).
	It is achieved by an iteration argument. We use the Hodge decomposition of $A$:
	$$A=A^{cf}+A^{df} = -|\nabla|^{-2} \nabla \,div \,A + A^{df}\, . $$
	We define $V_1 := -|\nabla|^{-2} \,div \,A(0)$ , so that $\nabla V_1 = A^{cf}(0)$. Thus $$ \|V_1\|_X := \| \nabla V_1\|_{H^s} = \|A^{cf}(0)\|_{H^s} \le \epsilon \, . $$ We define $U_1 := \exp(V_1)$ and consider the gauge transformation $T_1$ with
	\begin{align*}
		A_0 & \longmapsto U_1 A_0 U_1^{-1} - (\partial_t U_1) U_1^{-1} \\
		A & \longmapsto U_1 A U_1^{-1} - (\nabla U_1) U_1^{-1} \\
		\psi_{\pm} & \longmapsto U_1 \psi_{\pm} \, .
	\end{align*}
	Then $T_1$ preserves the temporal gauge, because $U_1$ is independent of $t$ .
	We obtain by Sobolev
	$$\|(T_1 \psi_{\pm})(0)\|_{H^l} \lesssim \|\exp V_1\|_X \|\psi_{\pm}(0)\|_{H^l} \lesssim (1+\epsilon) \|\psi_{\pm}(0)\|_{H^l} \, . $$
	Iteratively we define for $k \ge 2$ :
	$\nabla V_k := (T_{k-1}A)^{cf}(0)$ and $U_k := \prod_{l=k}^1 \exp V_l ,$ so that as in \cite{P}, Lemma 4.1 we obtain 
	$$\|V_k\|_X = \|(T_{k-1} A)^{cf}(0)\|_{H^s} \lesssim \epsilon^{\frac{k+1}{2}} \quad \forall k \ge 2 \, $$
	and $\|U_k-I\|_X \lesssim \epsilon$ (thus $\|U_k\|_X \lesssim 1$) .
	Let the gauge transformation $T_k$ be defined by
	\begin{align*}
		A_0 & \longmapsto U_k A_0 U_k^{-1} - (\partial_t U_k) U_k^{-1} \\
		A & \longmapsto U_k A U_k^{-1} - (\nabla U_k) U_k^{-1} \\
		\psi_{\pm} & \longmapsto U_k \psi_{\pm}  \, .
	\end{align*}
	This implies
	$$ \|(T_k \psi_{\pm})(0) \|_{H^l} \lesssim \|\exp V_k\|_X \|\psi_{\pm}(0)\|_{H^l} \lesssim (1+\epsilon) \|\psi_{\pm}(0)\|_{H^l} $$
	independently of $k$ .
	As in \cite{P}, Lemma 4.1 this allows to define a gauge transformation $T$ by
	$TA := \lim_{k \to \infty} T_k A$ in $C^0([0,1];H^s)$ , $\partial_t TA := \lim_{k \to \infty} \partial_t T_k A$ in $C^0([0,1];H^{s-1})$ and $T \psi_{\pm} := \lim_{k \to \infty} T_k \psi_{\pm}$ in $C^0([0,1];H^l)$ , which fulfills
	$(TA)^{cf}(0)=0$ .
	 We also deduce
	$$ \|(TA)^{df}(0)\|_{H^s} + \|(\partial_t TA)^{df}(0)\|_{H^{s-1}} + \|(T\psi_{\pm})(0)\|_{H^l}  \lesssim \epsilon \,  $$
 and  $$TA=U A U^{-1} -\nabla U U^{-1} \, ,  \, T\psi_{\pm} = U\psi_{\pm} \, ,$$
 where $U = \prod_{l=\infty}^1 \exp V_l$ , $U^{-1} = \prod_{l=1}^{\infty} \exp(-V_l)$ and the limits are taken with respect to $\|\cdot\|_X$ . It has the property $\|U\|_X = \|\nabla U\|_{H^s} \lesssim 1$ . 
 
 We want to show that $T$ preserves the regularity. That $TA$ has the same regularity was shown in \cite{P}, Lemma 4.1.  Let now $\chi=\chi(t)$ be a smooth function with $\chi(t)=1$ for $0 \le t \le 1$ and $\chi(t)=0$ for $t\ge 2$. We obtain :
	\begin{align*}
	\|U \psi_{\pm}\|_{X^{l,\half+}_{\pm}[0,1]} \lesssim 	\|U \psi_{\pm} \chi\|_{X^{l,\half+}_{\pm}} & \lesssim \|\nabla U \chi\|_{X^{s,1}_{\pm}} \|\psi_{\pm}\|_{X^{l,\half+}_{\pm}} \lesssim \|\nabla U\|_{H^s} \|\psi_{\pm}\|_{X^{l,\half+}_{\pm}} \\
		&\lesssim  \|\psi_{\pm}\|_{X^{l,\half+}_{\pm}} < \infty\, 
	\end{align*}
Here we applied the estimate 
$$\|uv\|_{X^{l,\half+\epsilon}_{\pm}} \lesssim \|\nabla u\|_{X^{s,1}_{\pm}} \|v\|_{X^{l,\half+\epsilon}_{\pm}} \,,$$ 
provided $ s > \half$ ,
for the second step, which is proved as \cite{P}, Lemma 4.2. Thus the regularity of $\psi_{\pm}$ is also preserved.

The inverse $T^{-1}$ , defined by
 $$T^{-1}B=U^{-1} B U + U^{-1}\nabla U \, ,  \, T^{-1}\psi_{\pm} = U^{-1}\psi_{\pm} \,, $$
has the same properties as $T$ .
\end{proof}

\section{Proof of Theorem \ref{Theorem1.1}}
\begin{proof}
	It suffices to construct a unique local solution of (\ref{6}),(\ref{7}),(\ref{8}) with initial conditions
	$$ A^{df}(0) = a^{df}  \, , \, (\partial_t A^{df})(0) = {a'}^{df}  \, , \,
	A^{cf}(0) = a^{cf} \, , \,\psi(0)=\psi_0 \, , $$
	which fulfill
	$$ \|A^{df}(0)\|_{H^s} + \|(\partial_t A)^{df}(0)\|_{H^{s-1}} + \|A^{cf}(0)\|_{H^s} + \|\psi(0)\|_{H^l} \le \epsilon $$
	for a sufficiently small $\epsilon > 0$.
	By Lemma \ref{Lemma} there exists a gauge transformation $T$ which fulfills the smallness condition (\ref{T1}) and $(TA)^{cf}(0) =0$. We use Prop. \ref{Prop} to construct a unique solution $(\tilde{A},\tilde{\phi})$ of (\ref{6}),(\ref{7}),(\ref{8}) , where $\tilde{A}=\tilde{A}_+^{df} + \tilde{A}_-^{df} +\tilde{A}^{cf}$ and $\tilde{\psi} = \tilde{\psi}_+ + \tilde{\psi}_-$ ,  with data
	$$\tilde{A}^{df}(0)= (TA)^{df}(0) \, , \, (\partial_t \tilde{A})^{df}(0) = (\partial_t (TA)^{df})(0) \, , \, \tilde{A}^{cf}(0)= (TA)^{cf}(0)=0 \, ,$$ 
	$$ \, \tilde{\psi}(0) = (T\psi)(0) \, ,  $$
	with the regularity
	$$ \tilde{A}^{df}_{\pm} \in X^{s,\frac{3}{4}+}_{\pm}[0,1]  ,  \tilde{A}^{cf} \in X^{s+\frac{1}{4},\frac{1}{2}+}_{\tau=0}[0,1]  ,  \partial_t \tilde{A}^{cf} \in C^0([0,1],H^{s-1})  , 
	\tilde{\psi}_{\pm} \in  X^{l,\half+}_{\pm}[0,1] \, . $$
	This solution satisfies also $\tilde{A} \in C^0([0,1],H^s) \cap C^1[0,1],H^{s-1})$, $\tilde{\psi}\in C^0([0,1],H^l)$ .
	
	Applying the inverse gauge transformation $T^{-1}$ according to Lemma \ref{Lemma} we obtain a unique solution of (\ref{6}),(\ref{7}),(\ref{8}) with the required initial data and also the same regularity.

\end{proof}


\begin{thebibliography}{999999}
\bibitem[AFS]{AFS}
P. d'Ancona, D. Foschi, and S. Selberg: 
{\sl Atlas of products for wave-{S}obolev spaces on
  {$\mathbb{R}^{1+3}$}}.
  Trans. Amer. Math. Soc. 364, (2012), 31-63.
\bibitem[AFS1]{AFS1}
P. d'Ancona, D. Foschi, and S. Selberg: {\sl Null structure and almost optimal regularity for
the Dirac-Klein-Gordon system}. J. European Math. Soc. 9 (2007), 877-899
\bibitem[CC]{CC} Y. Choquet-Bruhat and D. Christodoulou: {\sl Existence of global solutions of the Yang-Mills, Higgs and
spinor field equations in 3 + 1 dimensions}.
Annales scientifiques de l’É.N.S. 4e série, 14, no 4 (1981), 481-506  
\bibitem[GV]{GV} J. Ginibre and G. Velo: {\sl Generalized Strichartz inequalities for the wave equation.} J. Functional Anal. 133 (1995), 60-68   
\bibitem[HO]{HO} H. Huh and S.-J. Oh: {\sl Low regularity solutions to the Chern-Simons-Dirac and the Chern-Simons-Higgs equations in the Lorenz gauge}. Communications in Partial Differential Equations, 41 (2016),  375-397
\bibitem[KM]{KM}  S. Klainerman and M. Machedon: {\sl Finite energy solutions of the Yang-Mills equations in ${\mathbb R}^{3+1}$}. Ann. Math. 142 (1995), 39-119 
\bibitem[KM1]{KM1} S. Klainerman and M. Machedon: {\sl On the Maxwell-Klein-Gordon equation with finite energy}.
Duke Math. J. 74 (1994), 19–44.
\bibitem[KM2]{KM2} S. Klainerman and M. Machedon: {\sl Estimates for null forms and the spaces $H_{s,\delta}$}. Int. Math. Res. Notices
	1996, no. 17, 853-865
\bibitem[KMBT]{KMBT} S. Klainerman and M. Machedon (Appendices by J. Bourgain and  D. Tataru): {\sl Remark on Strichartz-type inequalities}. Int. Math. Res. Notices 1996, no.5, 201-220 
\bibitem[KS]{KS} S. Klainerman and S. Selberg: {\sl  Bilinear estimates and applications to nonlinear wave equations}. Communications in Contemporary Mathematics. 4 (2002) 223-295.
\bibitem[O]{O} S.-J. Oh: {\sl Gauge choice for the Yang-Mills equations using the Yang-Mills heat flow and local well-posedness in $H^1$}. J. Hyperbolic Differ. Equ. 11 (2014), 1-108. 
\bibitem[O1]{O1} S.-J. Oh: {\sl  Finite energy global well-posedness of the Yang-Mills equations on $\mathbb{R}^{1+3}$: an approach using the Yang-Mills heat flow.} Duke Math. J. 164 (2015), 1669-1732.
\bibitem[P]{P} H. Pecher: {\sl Local well-posedness for the $(n + 1)$-
	dimensional Yang–Mills and Yang–Mills–
	Higgs system in temporal gauge}. Nonlinear Differ. Equ. Appl. (2016) 23-40  
\bibitem[P1]{P1} H. Pecher: {\sl Local well-posedness of the coupled Yang-Mills and Dirac system for low regularity data}. Nonlinearity 35 (2022), 1–29 ,   Corrigendum:  Nonlinearity, to appear ,  	arXiv:2103.06770
\bibitem[Sz]{Sz} M. D. Schwartz: {\sl Quantum field  and the standard model}. Cambridge Univ. Press (2014)
\bibitem[SS]{SS} G. Schwarz and J. Sniatycki: {\sl  Gauge symmetries of an extended phase space for Yang-Mills and Dirac fields}. Ann. Inst. Henri Poincaré
66 (1997), 109-136
 \bibitem[ST]{ST} S. Selberg and A. Tesfahun: {\sl Null structure and local well-posedness in the energy class for the Yang-Mills equations in Lorenz gauge}. Journal of the European Mathematical Society 18 (2016), 1729-1752.
\bibitem[T]{T} T. Tao: {\sl Multilinear weighted convolutions of $L^2$-functions and applications to non-linear dispersive equations}. Amer. J. Math. 123 (2001), 838-908
\bibitem[T1]{T1} T. Tao: {\sl Local well-posedness of the Yang-Mills equation in the temporal gauge below the energy norm}. 	J. Diff. Equ.  189 (2003), 366-382
\bibitem[Te]{Te} A. Tesfahun: {\sl Local well-posedness of Yang-Mills equations in Lorenz gauge below the energy norm}. Nonlin. Diff. Equ. Appl. 22 (2015), 849-875


\end{thebibliography}
\end{document}